\documentclass[12pt, reqno]{amsart}
\usepackage[margin=1.25in]{geometry}        
\usepackage{graphicx}
\usepackage[section]{placeins}
\usepackage{flafter}
\makeatletter
\def\fps@figure{htbp}
\makeatother
\usepackage{amssymb}
\usepackage{mathtools}
\usepackage{xurl} 
\usepackage{hyperref}
\usepackage{etoolbox}

\makeatletter
\patchcmd{\@setauthors}{\MakeUppercase{\authors}}{%
  \MakeUppercase{\authors}\par\nobreak\smallskip
  {\normalfont\footnotesize\authoraffiliation}%
}{}{\PackageError{version1}{Unable to add the title-page affiliation}{Check the amsart author layout.}}
\makeatother

\usepackage{aliascnt}
\usepackage{doi}

\usepackage{epstopdf}

\usepackage[dvipsnames]{xcolor}
\usepackage{tikz,pgfplots}
\usetikzlibrary{patterns.meta}
\usepgflibrary{shadings}
\usepgflibrary {shadings}
\usetikzlibrary {arrows.meta}

\newcommand\myarrowR[2]{
\draw[arrows = {-Stealth[length=7pt, inset=4pt, width=7pt]},  line width=0.7pt] #1 to[bend right] #2;}
\newcommand\myarrowL[2]{
\draw[arrows = {-Stealth[length=7pt, inset=4pt, width=7pt]},  line width=0.7pt] #1 to[bend left] #2;}

\definecolor{ballcolor}{HTML}{85B8E8} 
\definecolor{simplexcolor}{HTML}{DE6C6C}  
\definecolor{simplexpicturecolor}{HTML}{C73E3E} 
\definecolor{twoptcolor}{HTML}{6ECA4F} 
\definecolor{twoptcolorshading}{HTML}{ADE48B} 
\definecolor{symmbr}{HTML}{FFF594} 
\definecolor{mygray}{HTML}{EEEEEE} 
\definecolor{mymidgray}{HTML}{9E9E9E} 
\definecolor{mypurple}{HTML}{5D3A9B} 
\definecolor{myorange}{HTML}{E66100} 

\newtheorem{theorem}{Theorem}[section]

\newaliascnt{corx}{thmx}

\aliascntresetthe{corx}

\newaliascnt{lemma}{theorem}
\newtheorem{lemma}[lemma]{Lemma}
\aliascntresetthe{lemma}

\newaliascnt{proposition}{theorem}

\aliascntresetthe{proposition}

\newaliascnt{corollary}{theorem}
\newtheorem{corollary}[corollary]{Corollary}
\aliascntresetthe{corollary}

\newaliascnt{conjecture}{theorem}
\newtheorem{conjecture}[conjecture]{Conjecture}
\aliascntresetthe{conjecture}

\newaliascnt{example}{theorem}

\aliascntresetthe{example}

\newaliascnt{question}{theorem}

\aliascntresetthe{question}

\theoremstyle{definition}
\newtheorem*{definition*}{Definition}
\newtheorem*{example*}{Example}
\newtheorem*{examples*}{Examples}

\newcommand{\B}{{\mathbb B}}

\newcommand{\R}{{\mathbb R}}

\newcommand{\Rn}{{{\mathbb R}^n}}
\newcommand{\Sph}{{\mathbb S}}

\DeclareMathOperator{\capzero}{Cap_0}

\DeclareMathOperator{\capp}{Cap_\mathit{p}}
\DeclareMathOperator{\capq}{Cap_\mathit{q}}

\DeclareMathOperator{\capr}{C_{\mathit{r}}}
\DeclareMathOperator{\caps}{C_{\mathit{s}}}

\def\ddefloop#1{\ifx\ddefloop#1\else\ddef{#1}\expandafter\ddefloop\fi}
\def\ddef#1{\expandafter\def\csname bb#1\endcsname{\ensuremath{\mathbb{#1}}}}
\ddefloop ABCDEFGHIJKLMNOPQRSTUVWXYZ\ddefloop

\def\ddef#1{\expandafter\def\csname c#1\endcsname{\ensuremath{\mathcal{#1}}}}
\ddefloop ABCDEFGHIJKLMNOPQRSTUVWXYZ\ddefloop

\title{Riesz capacity ratios with negative exponents}

\author{Qiuling Fan}
\newcommand{\authoraffiliation}{University of Illinois Urbana-Champaign}
\email{qfan0911@gmail.com}


\begin{document}

\begin{abstract}
We investigate sharp inequalities for ratios of Riesz capacities with negative exponents by combining computational experiments with rigorous analysis. For finite subsets of the line, we prove positivity of equilibrium masses when $-1<p<0$, enabling numerical tests of conjectured extremal ratios. In the plane, comparisons of the disk with regular polygon vertex sets reveal a cascade of transitions among the tested competitors and suggest a precise conjecture for the equilibrium measure of odd polygons, for which we give a partial proof. Numerical intersections of equality curves show that the regions where these sets outperform the disk are not simply nested. Similar numerical intersections occur in three dimensions between the regular-simplex equality curve and those of explicit five-point and six-point configurations. Motivated by the dimensional dependence of these comparisons, we prove that for each fixed $p<-2<q<0$, the regular simplex has a larger capacity ratio than the ball in all sufficiently large dimensions. Accompanying Python and Mathematica code supports reproduction and further testing of the conjectures.
\end{abstract}

\maketitle

\section{Riesz capacities and ratio conjectures}\label{chp:7}

Riesz capacity measures the size of a set in $\bbR^n$ through the pairwise interaction kernel $|x-y|^{-p}$, with exponent $p<n$. Which shapes maximize a ratio of capacities, and how do these shapes change as the two exponents vary? Continuing the capacity-ratio investigations of Clark and Laugesen~\cite{CL25}, we study negative exponents, with particular attention to symmetry breaking: parameter regions where the ball fails to maximize the ratio. Numerical experiments identify candidate extremizers and transition patterns, and guide the formulation of conjectures and the search for proofs.

In one dimension, we prove that every point of a finite set carries positive equilibrium mass when $-1<p<0$ (\autoref{thm:positivity}). This theorem provides the basis for our capacity computations. Experiments with discrete sets then supply numerical evidence for the interval and two-point extremal-ratio conjectures of Clark and Laugesen, stated in \autoref{chp:1d}.

In two dimensions, we build on the symmetry breaking established by Clark and Laugesen~\cite[Theorem~7]{CL25} by comparing the disk with regular polygon vertex sets. The experiments reveal a cascade of transitions among the tested competitors. They also motivate \autoref{conj:n-gon}, which specifies the maximizing triple of vertices for odd polygons when $p<-2$, beyond Bj\"orck's known at-most-three-point support bound. We prove the symmetric special case and test the general conjecture by exhaustive comparison of vertex triples on the plotting grids. The resulting equality curves with the disk intersect, showing that the regions where polygons outperform the disk are not simply nested as the number of vertices increases.

In higher dimensions, comparisons of the regular simplex and the ball suggest a limiting symmetry-breaking region. We prove that, for each fixed pair $p<-2<q<0$, the regular simplex has a larger capacity ratio than the ball in all sufficiently large dimensions; see \autoref{lemma:3dto-2} and \autoref{lemma:3dto0}. In three dimensions, explicit five-point and six-point configurations provide further symmetry-breaking examples. Numerically, their equality curves with the ball are each crossed by the regular-simplex curve, as shown in \autoref{fig:capratio_comparison_higher_dim}. These crossings echo the planar behavior: using configurations with more points does not uniformly enlarge the region where a configuration beats the ball.

We distinguish throughout between proved results, conjectural formulas, and finite numerical evidence. Winning among the tested configurations does not establish global optimality among compact sets, and verification on a parameter grid does not prove the odd-polygon conjecture. The accompanying \href{https://github.com/vhdvhd/Riesz-capacity-ratios-with-negative-exponents}{Python and Mathematica code} allows the experiments to be reproduced and extended. The transition patterns and unresolved conjectures provide concrete questions for further theoretical investigation.

\paragraph{\emph{Outline of the paper}} In \autoref{sec:7.1} we define the Riesz capacity and then \autoref{sec:7.2} introduces the maximization problem for the ratio of capacities. \autoref{chp:8} summarizes the numerical investigations that support the conjectures along with some new theoretical results. \autoref{chp:9} introduces notation that is helpful when $p$ is negative. The following three sections present the research in detail in one, two and higher dimensions respectively.

\subsection{Capacity and energy}\label{sec:7.1}

Let $K$ be a nonempty compact subset of $\bbR^n, n \ge 1$.

The Riesz $p$-capacity of $K$ when $p \neq 0$ is
\[
\capp(K) = \max_\mu \left( \int_K \! \int_K |x-y|^{-p} \, d\mu(x) d\mu(y) \right)^{\! \! -1/p},
\]
where $\mu$ ranges over the probability measures on $K$. When $p=0$, we write $\capzero(K)$ for the logarithmic capacity
\[
\capzero(K) = \max_\mu \, \exp \left( \int_K \! \int_K \log |x-y| \, d\mu(x) d\mu(y) \right).
\]
A measure $\mu$ that achieves the maximum is called a $p$-equilibrium measure. Intuitively, the equilibrium measure should spread out in some fashion across the set $K$ in order to achieve the maximum in the definition. The extent to which it spreads out generally depends on the exponent $p$. \\

In practice, it is often more convenient to work with the energy, rather than the capacity. The energy is \[
V_p(K) = 
\begin{cases}
\min_\mu \int_K \! \int_K |x-y|^{-p} \, d\mu(x) d\mu(y) , & p > 0 , \\
\max_\mu \int_K \! \int_K |x-y|^{-p} \, d\mu(x) d\mu(y) , & p < 0 .
\end{cases}
\]

At $p=0$ we use instead the logarithmic energy
\[
V_{\log}(K)
:= \min_{\mu}\;\iint_{K\times K} \log\!\frac{1}{|x-y|}\, d\mu(x)\,d\mu(y),
\]
where the minimum is taken over probability measures $\mu$ supported on $K$.\\

Expressed in terms of the energy, the Riesz $p$-capacity of $K$ is
\begin{equation*} 
\capp(K) = 
\begin{cases}
V_p(K)^{-1/p} , & p \neq 0 , \\
\exp(-V_{\log}(K)) , & p = 0 .
\end{cases}
\end{equation*}

\paragraph{\emph{Remark}} Riesz energy is the prototypical example of a pairwise interaction energy. The exponent $p$ controls the strength of the repulsive singularity (when $p>0$). The case when $p=1$ in three dimensions is the classical electrostatic energy. In recent years, many authors have investigated Riesz energy and capacity and their discrete analogues, for example, by examining convergence rates as the number of points tends to infinity. The book by Borodachov, Hardin and Saff \cite{BHS19} provides an excellent account of this theory. In a different direction, optimal configurations and symmetry breaking for Riesz capacity under diameter constraints have been studied by Burchard \cite{BCH20} and Lim and McCann \cite{LM21}. 

\subsection{Riesz capacity ratio: introducing the maximization problem}\label{sec:7.2}

For fixed $n,p,q$, we seek the supremum of the capacity ratio over compact $K\subset\bbR^n$ with $\capp(K)>0$, a restriction understood throughout:
\[
\sup_{K\subset \bbR^n}\;\frac{\capq(K)}{\capp(K)}\,.
\]

By definition, the capacity ratio is invariant under translations and rotations of $K$. It is also scale-invariant, since for every $s>0$,
\[
\frac{\capq(sK)}{\capp(sK)}=\frac{\capq(K)}{\capp(K)},
\]
because $\capp(sK)=s\capp(K)$. Thus the ratio depends only on the shape of $K$, not on its size or orientation.

The Riesz capacity–ratio optimization problem introduced by Clark and Laugesen~\cite{CL25} for $-\infty < p,q<n$ unifies many classical results and conjectures. A notable instance is the P\'olya--Szeg\H{o} conjecture~\cite[p.~140]{S56}, which seeks the planar set maximizing the ratio of Newtonian to logarithmic capacity and corresponds to the case $(p,q)=(0,1)$, with a natural generalization to $0<p<q$. See~\cite{CL25} for more related theorems. In these problems, the ball is the known—or conjectured—maximizer in many regions of the $(p,q)$ plane, but in other regions, the regular simplex is conjectured to be the maximizer. This leads to symmetry-breaking behavior across a transition region.

The capacity ratio is interesting especially because geometric quantities like volume and diameter are limiting cases of capacity. Clark and Laugesen~\cite[Theorem~1]{CL25} identified the limiting behaviors as

\begin{equation}\label{eq:limits}
\lim_{p\nearrow n} \frac{\capp(K)^p}{n-p} = \frac{\operatorname{Vol}_n(K)}{|\bbS^{n-1}|}
\quad\text{and}\quad
\lim_{p\to -\infty} \capp(K) = \operatorname{diam}(K),
\end{equation}
for compact $K\subset\R^n$. Thus the capacity interpolates between volume (as $p\to n$) and diameter (as $p\to -\infty$). In the overview diagrams, the boundary labels $n$ and $-\infty$ denote the limiting functionals $\operatorname{Vol}_n(K)^{1/n}$ and $\operatorname{diam}(K)$, respectively, with a positive denominator understood. See \autoref{fig:ballplots} for illustration. Theorem~1.1 in Fan and Laugesen \cite{FL24} extends the volume-side limit to compact strongly $d$-rectifiable sets $E\subset\bbR^n$, where $1\le d\le n$ is an integer, including compact submanifolds:
\[
\lim_{p\nearrow d} \frac{\capp(E)^p}{d-p} = \frac{\cH^d(E)}{|\Sph^{d-1}|}\,.
\]

\begin{figure}[htp]
\begin{center}
\includegraphics[width=0.5\textwidth]{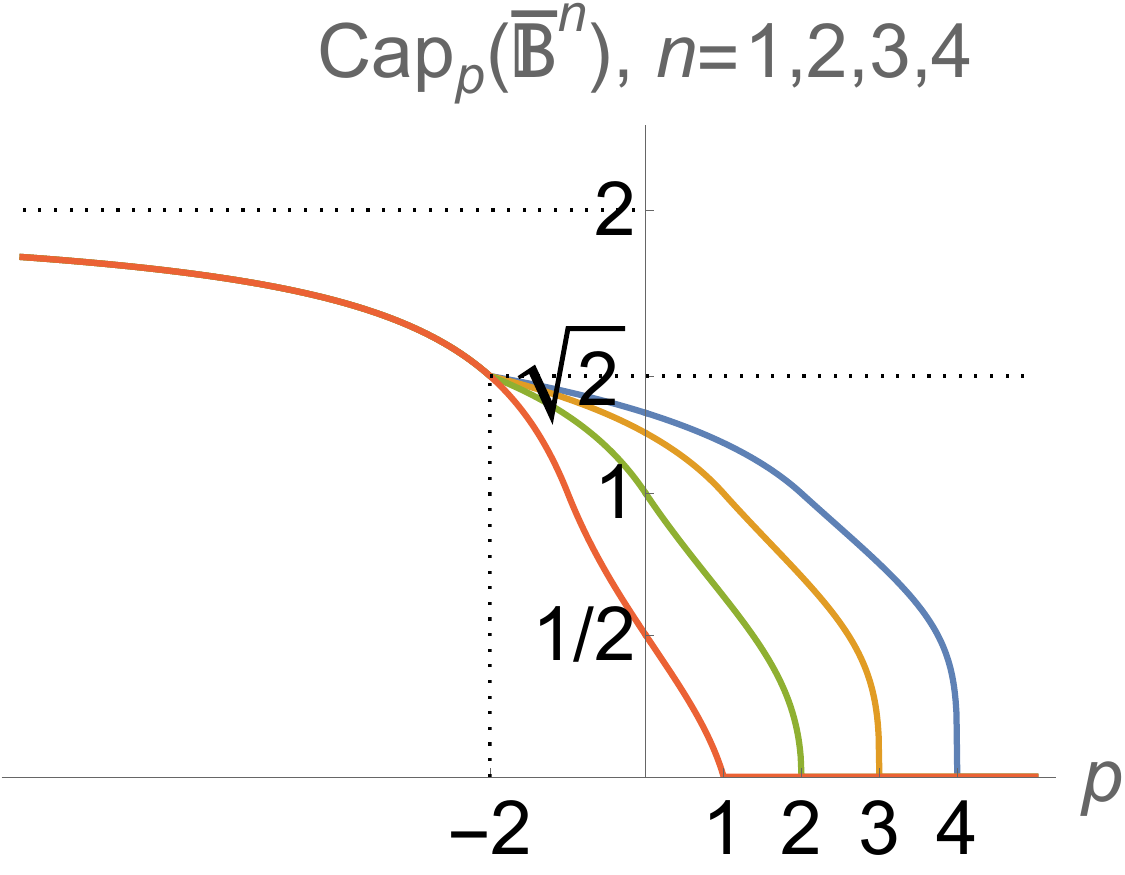}
\end{center}
\caption{\label{fig:ballplots} The Riesz capacity $\capp(\overline{\B}^n)$ of the closed unit ball  as a function of $p$, for four values of $n$. The intercept at $p=0$ is the logarithmic capacity $\capzero(\overline{\B}^n)$. When $p \leq -2$, the capacity of the ball is $2^{1+1/p}$ for each $n$, with limiting value $2$ (the diameter of the ball) as $p \to -\infty$. Figure credit: from Clark and Laugesen \cite[Figure 1]{CL24b}.
}
\end{figure}

We concentrate on the regime $p,q<0$, for two reasons. 
First, the absence of singularities in the kernel makes the analysis technically simpler and allows us to treat discrete sets. 
Second, this negative-exponent regime exhibits symmetry breaking, so the optimal shape need not be a ball. Clark and Laugesen~\cite{CL25} were the first to study the regime $p,q<0$. They showed that with negative exponents the optimal shape need not be a ball. Depending on $(p,q)$, the maximizer can be the endpoints of an interval ($n=1$), the vertices of a regular simplex ($n\ge 2$), or other configurations. See \autoref{fig:pqdiagram1D}, \autoref{fig:pqdiagram2D} and \autoref{fig:pqdiagram3D} for the cases $n=1,2$ and $n\ge3$.

%
%


\section{\bf Summary of results}\label{chp:8}


Our aim is to examine the dashed-outlined regions in \autoref{fig:pqdiagram1D}, \autoref{fig:pqdiagram2D} and \autoref{fig:pqdiagram3D}, with $p,q<0$ where the optimizer for the Riesz capacity ratio remains open in one, two, and higher dimensions. We combine numerical experiments that support the conjectures with some theorems that clarify the shape of potential optimizers. In addition, we identify families of shapes that in two and higher dimensions attain larger ratios than the ball, thus proving that symmetry breaking occurs.

\subsection{Overview in one dimension (see \autoref{fig:pqdiagram1D})}


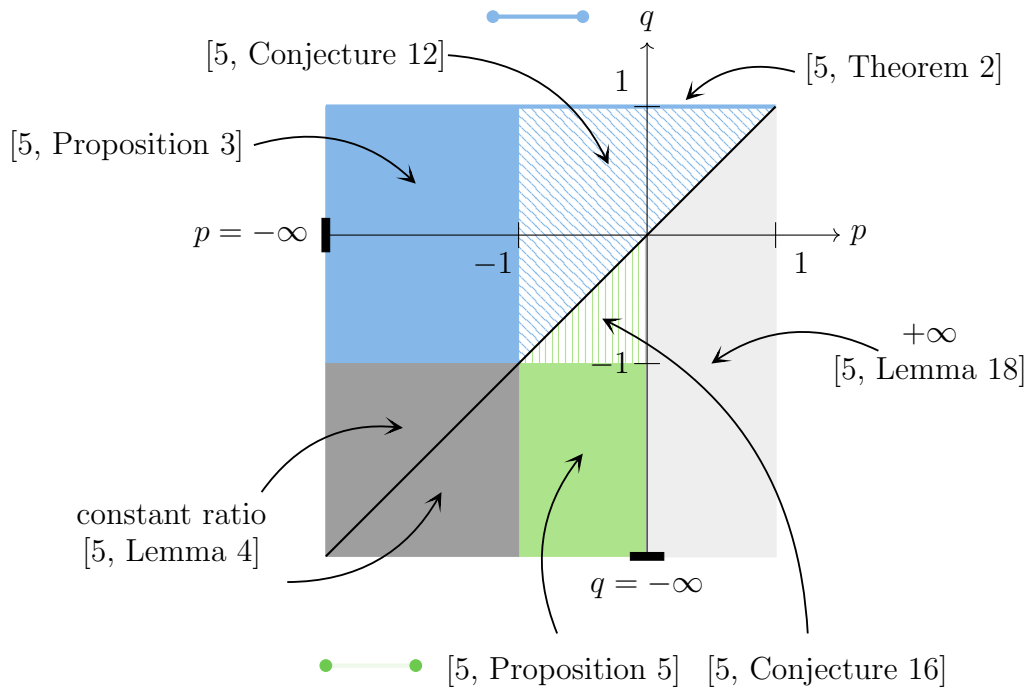
\begin{figure}[!htb]
\begin{center}
\begin{tikzpicture}[scale=1.7]

\def\n{1}
\def\pmax{1.5}
\def\pmin{-2.5}
\def\qmax{1.5}
\def\qmin{-2.5}


\filldraw[pattern={Lines[angle=-45,distance=2.5pt]}, pattern color=ballcolor, draw=none] (0,0) -- (\n, \n) -- (-1, \n) -- (-1, -1) -- cycle;
\filldraw[ballcolor , draw=ballcolor] (\pmin, \n) -- (-1, \n) -- (-1, -1) -- (\pmin,-1) -- cycle;
\filldraw[pattern={Lines[angle=90,distance=2.5pt]}, pattern color=twoptcolorshading, draw=none] (0,0) -- (-1, -1) -- (0, -1);
\filldraw[color=twoptcolorshading, draw=twoptcolorshading] (-1,\qmin) -- (0,\qmin) -- (0,-1) -- (-1,-1) -- cycle;
\filldraw[color=mymidgray, draw=mymidgray] (\pmin,\qmin) -- (-1,\qmin) -- (-1,-1) -- (\pmin,-1) -- cycle;
\filldraw[color=mygray, draw=mygray,  thick] (\n,\n) -- (\n,\qmin) -- (0,\qmin) -- (0,0) -- cycle;

\draw[ballcolor, ultra thick]  (\pmin,\n) -- (\n,\n);


\draw (\pmin,1.6)   node[below] {\cite[Conjecture 12]{CL25}};
\myarrowL{(-\n-0.55,1.4)}{(-.3,\n-0.45)}
\draw (\pmin-1.55,0.9)   node[below] {\cite[Proposition 3]{CL25}};
\myarrowL{(\pmin-0.56,0.7)}{(\pmin+0.7,0.4)}
\draw (\n+1,1.5)   node[below] {\cite[Theorem 2]{CL25}};
\myarrowR{(\n+0.15,1.27)}{(\n-0.7,\n+0.05)}
\def\tpx{-1.2} 
\def\tpy{\n+0.7} 
\def\tpl{0.7} 
\draw[ballcolor, ultra thick] (\tpx,\tpy) -- (\tpx+\tpl,\tpy);
\filldraw[ballcolor] (\tpx,\tpy) circle (0.25ex);
\filldraw[ballcolor] (\tpx+\tpl,\tpy) circle (0.25ex);

\draw (1.4,\qmin-1.1)   node[above] {\cite[Conjecture 16]{CL25}};
\myarrowR{(1.25,\qmin-0.6)}{(-0.35,-0.65)}
\draw (-0.65,\qmin-1.1)   node[above] {\cite[Proposition 5]{CL25}};
\myarrowL{(-0.7,\qmin-0.6)}{(-0.5,-1.7)}
\def\tpx{-2.5} 
\def\tpy{\qmin-0.85} 
\def\tpl{0.7} 
\draw[twoptcolor!12, ultra thick] (\tpx,\tpy) -- (\tpx+\tpl,\tpy);
\filldraw[twoptcolor] (\tpx,\tpy) circle (0.25ex);
\filldraw[twoptcolor] (\tpx+\tpl,\tpy) circle (0.25ex);

\draw (\pmin-1.2,-2) node[below,align=center] {constant ratio\\\cite[Lemma 4]{CL25}};
\myarrowL{(\pmin-0.5,-2)}{(\pmin+0.6,-1.5)}
\myarrowR{(\pmin-0.3,-2.7)}{(\pmin+0.9,-2)}

\draw (\n+1.2,-0.6)  node[below,align=center] {$+\infty$\\ \cite[Lemma 18]{CL25}};
\myarrowR{(\n+0.8,-0.9)}{(\n-0.5,-1)}


\draw[black,  thick]  (\qmin,\pmin) to (\n,\n);
\draw[->] (\pmin,0) to (\pmax,0);
\draw (\pmax,0)  node[right] {$p$};
\draw[->] (0,\qmin) to (0,\qmax);
\draw (0,\qmax)  node[above] {$q$};
\draw[-] (\n,-0.1) to (\n,0.1);
\draw (\n+0.2, -0.05) node[below] {$\n$};
\draw[-] (-1,-0.1) to (-1,0.1);
\draw (-1.2, -0.05) node[below] {$-1$};
\draw[-] (-0.1,\n) to (0.1,\n);
\draw (-0.05,\n+0.2) node[left] {$\n$};
\draw[-] (-0.1,-1) to (0.1,-1);
\draw (-0.05,-1) node[left] {$-1$};
\draw (\pmin-0.05,0)  node[left] {$p=-\infty$};
\draw[black,fill=black] (\pmin-0.03,-0.13)  rectangle  (\pmin+0.03,0.13) ;
\draw (0,\qmin-0.05)  node[below] {$q=-\infty$};
\draw[black,fill=black] (-0.13,\qmin-0.03)  rectangle  (0.13,\qmin+0.03) ;
\end{tikzpicture}
\caption{\label{fig:pqdiagram1D} ($n=1$) Maximizing $\capq(K)/\capp(K)$ for $K \subset \R$. Solid regions and segments on the boundary indicate rigorous results. Striped regions are conjectural. Blue corresponds to the interval and green to the two-point set. In \autoref{chp:1d}, we provide numerical support for the conjectures in the square where $-1 < p, q <0$. \textsc{Credit:} Figure adapted from Clark and Laugesen \cite[Figure 1]{CL25}.}
\end{center}
\end{figure}


In one dimension, we investigate the case when $p, q \in [-1, 0)$. Specifically, we provide numerical support for the two cases:

\begin{itemize}
	\item In the upper triangle (where $p<q<0$), the closed interval maximizes the capacity ratio among compact sets, see \autoref{thm:1d-interval}.
	\item In the lower triangle (where $q<p<0$), any set with exactly two points maximizes the capacity ratio, see \autoref{thm:1d-2pt}.
\end{itemize}

Also, we prove a positivity result \autoref{thm:positivity} that the mass is positive at each point in a finite set, and  use this theorem to simplify the algorithm we use to compute the Riesz capacity.

\FloatBarrier 
\subsection{Overview in two dimensions (see \autoref{fig:pqdiagram2D})}


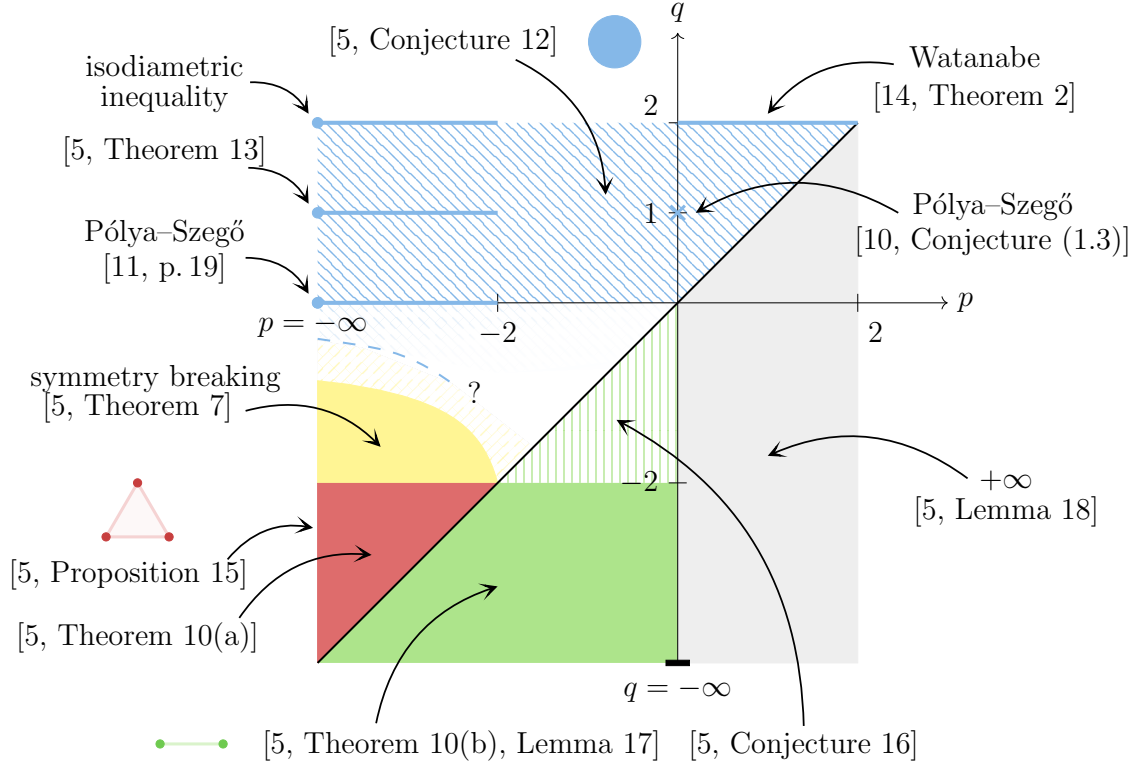
\begin{figure}[!htb]
\centering
\resizebox{\linewidth}{!}{%
\begin{tikzpicture}[scale=1.25]

\def\n{2}
\def\pmax{3}
\def\pmin{-4}
\def\qmax{3}
\def\qmin{-4}


\filldraw[top color=ballcolor!40, bottom color=white, middle color=ballcolor!5, shading angle=10, draw=none] (\pmin,\n-2.4) -- (\pmin,\n-2) -- (0,\n-2) -- (-0.75,-0.75) -- (-\n,-0.8);
\filldraw[pattern={Lines[angle=-45,distance=3pt,line width=2.4pt, xshift=2pt]}, pattern color=white, draw=none] (\pmin+0.02,\n-3) -- (\pmin-0.2,\n-2) -- (0,\n-2) -- (-0.65,-0.65);
\filldraw[pattern={Lines[angle=-45,distance=3pt,line width=2.4pt, xshift=2pt]}, pattern color=white, draw=none] (\pmin+0.02,\n-3) -- (\pmin-.2,\n-2) -- (0,\n-2) -- (-0.65,-0.65);
\draw[white, ultra thick] (\pmin,\n-2) -- (\pmin,\n-2.4) -- (-\n,-0.8) -- (-0.75,-0.75);
\filldraw[pattern={Lines[angle=-45,distance=3pt, line width=0.5pt]}, pattern color=ballcolor, draw=none] (0,0) -- (\n, \n) -- (\pmin, \n) -- (\pmin, \n-2) -- (0, \n-2);
\filldraw[color=mygray, draw=mygray] (\n,\n) -- (\n,\qmin) -- (0,\qmin) -- (0,0);
\filldraw[pattern={Lines[angle=90,distance=3pt, line width=0.5pt]}, pattern color=twoptcolorshading, draw=none] (-2,-2) -- (0, -2) -- (0, 0)  ;
\filldraw[color=twoptcolorshading, draw=none] (\pmin,\qmin) -- (-2, -2) -- (0,-2) -- (0,\qmin)  ;
\filldraw[color=simplexcolor] (\pmin,\qmin) -- (\qmin, -2) -- (-2, -2) ;
\filldraw[top color=white, bottom color=symmbr, middle color=symmbr!60, shading angle=-29, draw=white] (\pmin,-0.9) -- (-2,-2) -- (-1.6,-1.6) -- (-2.5,-0.9) -- (-3.5,-0.5) -- (\pmin,-0.42);
\filldraw[pattern={Lines[angle=45,distance=3pt,line width=2.4pt, xshift=2pt]}, pattern color=white, draw=none]  (\pmin,-0.9) -- (-2,-2) -- (-1.6,-1.6) -- (-2.5,-0.9) -- (-3.5,-0.5) -- (\pmin,-0.42);
\filldraw[pattern={Lines[angle=45,distance=3pt,line width=2.4pt, xshift=2pt]}, pattern color=white, draw=none]  (\pmin,-0.9) -- (-2,-2) -- (-1.6,-1.6) -- (-2.5,-0.9) -- (-3.5,-0.5) -- (\pmin,-0.42);
\draw[white] (\pmin,-0.9) -- (\pmin,-0.42);
\filldraw[color=symmbr] plot[] coordinates {(-2., -2.) (-2.0184, -1.93555) (-2.03957, -1.87111) (-2.06393, -1.80666) (-2.09199, -1.74222) (-2.12439, -1.67778) (-2.16191, -1.61333) (-2.20551, -1.54889) (-2.25643, -1.48444) (-2.31628, -1.42) (-2.38713, -1.35555) (-2.47179, -1.29111) (-2.57407, -1.22666) (-2.69934, -1.16222) (-2.8554, -1.09778) (-3.05398, -1.03333) (-3.31366, -0.968887) (-3.66564, -0.904442) (-4, -0.865776) (-4,-2)};


\draw[black,  thick]  (\qmin,\pmin) to (\n,\n);

\draw[->] (\pmin,0) to (\pmax,0);
\draw (\pmax,0)  node[right] {$p$};

\draw[->] (0,\qmin) to (0,\qmax);
\draw (0,\qmax)  node[above] {$q$};

\draw[-] (\n,-0.1) to (\n,0.1);
\draw (\n+0.2, -0.1) node[below] {$\n$};
\draw[-] (-2,-0.1) to (-2,0.1);
\draw (-2, -0.1) node[below] {$-2$};

\draw[-] (-0.1,-2) to (0.1,-2);
\draw (-0.1,-2) node[left] {$-2$};
\draw[-] (-0.1,\n-1) to (0.1,\n-1);
\draw (-0.1,\n-1) node[left] {$1$};
\draw[-] (-0.1,\n) to (0.1,\n);
\draw (-0.1,\n+0.2) node[left] {$\n$};

\draw[ballcolor,fill=ballcolor] (\pmin,\n-2) circle (.32ex);

\draw (\pmin-0.05,0)  node[below] {$p=-\infty$};

\draw (0,\qmin-0.05)  node[below] {$q=-\infty$};
\draw[black,fill=black] (-0.13,\qmin-0.03)  rectangle  (0.13,\qmin+0.03) ;


\draw[-, ballcolor, ultra thick] (\n-2,\n) to (\n,\n);
\draw[-, ballcolor, ultra thick] (\pmin,\n) to (-2,\n);
\draw[-, ballcolor, ultra thick] (\pmin,\n-1) to (-2,\n-1);
\draw[-, ballcolor, ultra thick] (\pmin,\n-2) to (-2,\n-2);
\draw[ballcolor, dash pattern=on 5pt off 5.1pt, thick] (\pmin,\n-2.4)  .. controls (\pmin+0.65,\n-2.5) and (\pmin+0.95,\n-2.5) .. (\pmin+1.5,\n-2.9);
\draw  (\pmin+1.55,\n-3) node[right] {?};
\draw[ballcolor, very thick] (\n-2+0.08,\n-1+0.08) -- (\n-2-0.08,\n-1-0.08);
\draw[ballcolor, very thick] (\n-2-0.08,\n-1+0.08) -- (\n-2+0.08,\n-1-0.08);
\draw[ballcolor,fill=ballcolor] (\pmin,\n) circle (.32ex);
\draw[ballcolor,fill=ballcolor] (\pmin,\n-2) circle (.32ex);
\draw[ballcolor,fill=ballcolor] (\pmin,\n-1) circle (.32ex);


\draw (\pmin-1.8,-1.1) node[above] {symmetry breaking};
\draw (\pmin-2.0,-1.45) node[above] {\cite[Theorem 7]{CL25}};
\myarrowL{(\pmin-0.8,-1.2)}{(\pmin+0.7,-1.6)}

\draw (\pmin-2,\qmin+0.0)   node[above] {\cite[Theorem 10(a)]{CL25}};
\myarrowL{(\pmin-0.87,\qmin+0.5)}{(\pmin+0.6,\qmin+1.2)}
\draw (\pmin-2.1,\qmin+0.7)   node[above] {\cite[Proposition 15]{CL25}};
\myarrowL{(\pmin-1.0,\qmin+1.15)}{(\pmin-0.05,\qmin+1.5)}
\def\sx{(\pmin-2}
\def\sy{\qmin+2.0}
\draw[simplexcolor!30, very thick, fill=simplexcolor!5] (\sx,\sy) -- (\sx-0.35,\sy-0.6) -- (\sx+0.35,\sy-0.6) -- cycle;
\filldraw[simplexpicturecolor] (\sx,\sy) circle (0.25ex);
\filldraw[simplexpicturecolor] (\sx-0.35,\sy-0.6) circle (0.25ex);
\filldraw[simplexpicturecolor] (\sx+0.35,\sy-0.6) circle (0.25ex);

\draw (1.4,\qmin-1.2)   node[above] {\cite[Conjecture 16]{CL25}};
\myarrowR{(1.4,\qmin-0.7)}{(-0.68,\qmin+2.7)}
\draw (-2.4,\qmin-1.2)   node[above] {\cite[Theorem 10(b), Lemma 17]{CL25}};
\myarrowL{(-3.35,\qmin-0.67) }{(\qmin+2,\qmin+0.85)}
\def\tpx{-5.75} 
\def\tpy{\qmin-0.9} 
\def\tpl{0.7} 
\draw[twoptcolorshading!45, very thick] (\tpx,\tpy) -- (\tpx+\tpl,\tpy);
\filldraw[twoptcolor] (\tpx,\tpy) circle (0.25ex);
\filldraw[twoptcolor] (\tpx+\tpl,\tpy) circle (0.25ex);

\draw (\n+1.5,\n-1.2)  node[above] {P\'{o}lya--Szeg\H{o}};
\draw (\n+1.5,\n-1.6)  node[above] {\cite[Conjecture (1.3)]{PS45}};

\myarrowR{(\n+0.45,\n-1)}{(\n-1.8,\n-1)}
\draw (-4,\n+0.9)  node[right] {\cite[Conjecture 12]{CL25}};
\myarrowL{(-1.8,\n+0.7)}{(-0.8,\n-0.8)}
\draw (\pmin-1.7,\n-2+0.50) node[above] {P\'{o}lya--Szeg\H{o}};
\draw (\pmin-1.7,\n-2+0.1) node[above] {\cite[p.{\,}19]{PS51}};
\myarrowL{(\pmin-0.8,\n-1.5)}{(\pmin-0.1,\n-1.9)}
\draw (\pmin-1.7,\n-0.6) node[above] {\cite[Theorem 13]{CL25}};
\myarrowL{(\pmin-0.77,\n-0.5)}{(\pmin-0.1,\n-0.9)}
\draw (\pmin-1.7,\n+0.4) node[above] {isodiametric};
\draw (\pmin-1.7,\n) node[above] {inequality};
\myarrowL{(\pmin-0.75,\n+0.45)}{(\pmin-0.1,\n+0.1)}
\draw (\n+1.3,\n+0.5)  node[above] {Watanabe};
\draw (\n+1.3,\n)  node[above] {\cite[Theorem 2]{W83}};
\myarrowR{(\n+0.4,\n+0.65)}{(\n-1,\n+0.1)}
\filldraw[ballcolor] (-0.7,\n+0.9) circle (1.6ex);

\draw (\n+1.65,-2.2)  node[above] {$+\infty$ };
\draw (\n+1.65,-2.6)  node[above] {\cite[Lemma 18]{CL25}};
\myarrowR{(\n+1.22,-1.95)}{(\n-1,-1.7)}

\end{tikzpicture}
}

\caption{\label{fig:pqdiagram2D} ($n=2$) Maximizing $\capq(K)/\capp(K)$ for $K \subset \R^2$. Solid regions and solid segments indicate rigorous results. Striped regions and dashed curves are conjectural. Blue corresponds to the disk, red to the regular three-point set, and green to the two-point set.\\ In \autoref{chp:2d} we improve the symmetry breaking results in the region $p<-2$ and $-2<q<0$, through a combination of numerical and theoretical work. \textsc{Credit:} Figure adapted from Clark and Laugesen \cite[Figure 2]{CL25}.}
\end{figure}



For $p,q\in[-2,0)$, comparisons of the disk and regular polygon vertex sets in \autoref{sec:2d0to2} support the following conjectures for compact sets:
\begin{itemize}
	\item In the upper triangle (where $-2\le p<q<0$), the closed ball $\overline{\bbB^2}$ maximizes the capacity ratio among compact sets.
	\item In the lower triangle (where $-2\le q<p<0$), any set with exactly two points maximizes the capacity ratio.
\end{itemize}

The case $p\in(-\infty, -2]$ and $q \in [-2, 0)$ is considerably more complicated. Starting from the symmetry-breaking yellow region in \autoref{fig:pqdiagram2D} by Clark and Laugesen \cite[Theorem 7]{CL25}, we found a family of shapes that further improves on the capacity ratio of the disk: the vertex sets $P_N$ of regular $N$-gons where $N$ is odd. 

Specifically, we find a curve \autoref{ineq:region} (with equality) in the $pq$-plane where the capacity ratio of the vertex set of a regular $N$-gon and a disk are equal. The theoretical result is built upon a simply-stated but yet-to-be-proven \autoref{conj:n-gon} that when $p<-2$, an equilibrium measure of $P_N$ is supported on exactly three points: one point and the two most distant vertices from that point. Numerical experiments also support this conjecture.

Moreover, the symmetry-breaking curves comparing each $P_N$ with the disk intersect in \autoref{sec:intersection}, showing that their disk-improvement regions are not nested. The numerical winner maps separately show that the best tested $N$ varies with $p$ and $q$.

\FloatBarrier 
\subsection{Overview in higher dimensions (see \autoref{fig:pqdiagram3D})}


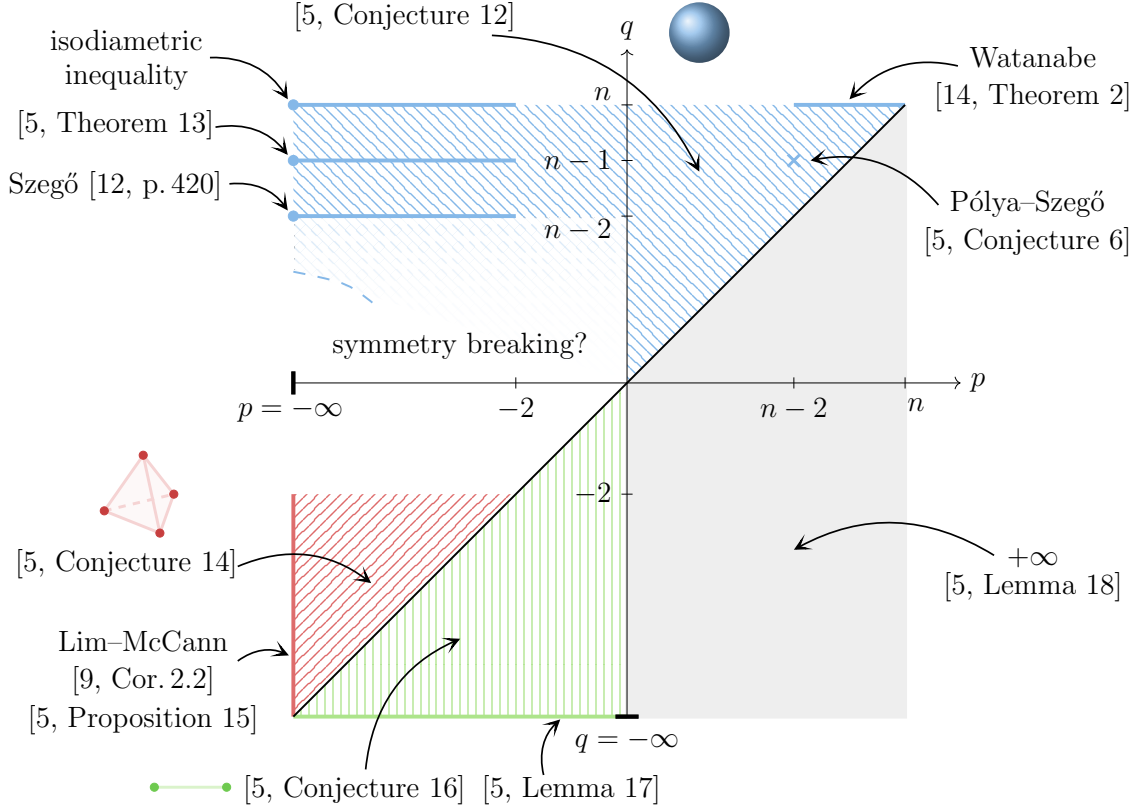
\begin{figure}[!htb]
\centering
\resizebox{\linewidth}{!}{%
\begin{tikzpicture}[scale=0.8]

\def\n{5}
\def\pmax{6}
\def\pmin{-6}
\def\qmax{6}
\def\qmin{-6}


\filldraw[top color=ballcolor!40, bottom color=white, middle color=ballcolor!5, draw=none] (\pmin+0.02,\n-3) -- (\pmin+0.02,\n-2.02) -- (0,\n-2.02) -- (0,0);
\filldraw[pattern={Lines[angle=-45,distance=3pt,line width=2.4pt, xshift=2pt]}, pattern color=white, draw=none] (\pmin+0.02,\n-3) -- (\pmin-0.2,\n-2) -- (0,\n-2) -- (0,0);
\filldraw[pattern={Lines[angle=-45,distance=3pt,line width=2.4pt, xshift=2pt]}, pattern color=white, draw=none] (\pmin+0.02,\n-3) -- (\pmin-.2,\n-2) -- (0,\n-2) -- (0,0);
\draw[white, ultra thick] (\pmin,\n-3) -- (0,0);
\filldraw[pattern={Lines[angle=-45,distance=3pt, line width=0.5pt]}, pattern color=ballcolor, draw=none] (0,0) -- (\n, \n) -- (\pmin, \n) -- (\pmin, \n-2.02) -- (0, \n-2.02) ;
\filldraw[color=mygray, draw=mygray, ultra thick] (\n,\n) -- (\n,\qmin) -- (0,\qmin) -- (0,0);
\filldraw[pattern={Lines[angle=90,distance=3pt, line width=0.5pt]}, pattern color=twoptcolorshading, draw=none] (\pmin,\qmin) -- (0, \qmin) -- (0, 0)  ;
\filldraw[pattern={Lines[angle=45,distance=3pt, line width=0.5pt]}, pattern color=simplexcolor, draw=none] (\pmin,\qmin) -- (\qmin, -2) -- (-2, -2)  ;


\draw[-, ballcolor, ultra thick] (\n-2,\n) to (\n,\n);
\draw[-, ballcolor, ultra thick] (\pmin,\n) to (-2,\n);
\draw[-, ballcolor, ultra thick] (\pmin,\n-1) to (-2,\n-1);
\draw[-, ballcolor, ultra thick] (\pmin,\n-2) to (-2,\n-2);
\draw[ballcolor,fill=ballcolor] (\pmin,\n) circle (.5ex);
\draw[ballcolor,fill=ballcolor] (\pmin,\n-2) circle (.5ex);
\draw[ballcolor,fill=ballcolor] (\pmin,\n-1) circle (.5ex);
\draw[-, ultra thick, simplexcolor] (\pmin,\qmin) -- (\pmin, -2);
\draw[twoptcolorshading, ultra thick] (0, \qmin ) -- (\pmin,\qmin);
\draw[ballcolor, dash pattern=on 5pt off 5.1pt, thick] (\pmin,\n-3)  .. controls (\pmin+1,\n-3.2)  .. (\pmin+1.59,\n-3.7);
\draw[ballcolor, very thick] (\n-2+0.11,\n-1+0.11) -- (\n-2-0.11,\n-1-0.11) ;
\draw[ballcolor, very thick] (\n-2-0.11,\n-1+0.11) -- (\n-2+0.11,\n-1-0.11) ;


\draw (\n+2.3,-3.5)  node[above] {$+\infty$};
\draw (\n+2.3,-4.1)  node[above] {\cite[Lemma 18]{CL25}};
\myarrowR{(\n+1.6,-3.1)}{(\n-2,-3)}

\draw (-3,0.25)  node[above] {symmetry breaking?};

\draw (\n+2.2,\n-2.2)  node[above] {P\'{o}lya--Szeg\H{o}};
\draw (\n+2.2,\n-2.9)  node[above] {\cite[Conjecture 6]{CL25}};
\myarrowR{(\n+0.6,\n-1.8)}{(\n-1.7,\n-1)}

\draw (-4,\n+1.15)  node[above] {\cite[Conjecture 12]{CL25}};
\myarrowL{(-2.43,\n+1.25)}{(\n-3.7,\n-1.2)}
\draw (\pmin-3,\n+0.8) node[above] {isodiametric};
\draw (\pmin-3,\n+0.1) node[above] {inequality};
\myarrowL{(\pmin-1.5,\n+0.85)}{(\pmin-0.1,\n+0.1)}
\draw (\pmin-3.2,\n-0.8) node[above] {\cite[Theorem 13]{CL25}};
\myarrowL{(\pmin-1.4,\n-0.4) }{(\pmin-0.1,\n-1+0.1)}
\draw (\pmin-3.2,\n-2+0.1) node[above] {Szeg\H{o} \cite[p.{\,}420]{S31}};
\myarrowL{(\pmin-1.15,\n-1.42)}{(\pmin-0.1,\n-2+0.1)}
\draw (\n+2.3,\n+0.5)  node[above] {Watanabe};
\draw (\n+2.3,\n-0.3)  node[above] {\cite[Theorem 2]{W83}};
\myarrowR{(\n+0.95,\n+0.7)}{(\n-1,\n+0.1)}
\shade[ball color=ballcolor] (1.3,\n+1.3) circle (3ex);

\draw (-4.9,\qmin-1.76)   node[above] {\cite[Conjecture 16]{CL25}};
\myarrowL{(-4.7,\qmin-0.95) }{(-3,\qmin+1.4)}
\draw (-1.0,\qmin-1.76)   node[above] {\cite[Lemma 17]{CL25}};
\myarrowL{(-1.5,\qmin-1.05)}{(-1.2,\qmin-0.05)}
\def\tpx{\pmin-2.5} 
\def\tpy{\qmin-1.27} 
\def\tpl{1.3} 
\draw[twoptcolorshading!45, very thick] (\tpx,\tpy) -- (\tpx+\tpl,\tpy);
\filldraw[twoptcolor] (\tpx,\tpy) circle (0.4ex);
\filldraw[twoptcolor] (\tpx+\tpl,\tpy) circle (0.4ex);

\draw (\pmin-3.0,\qmin+2.3)   node[above] {\cite[Conjecture 14]{CL25}};
\myarrowL{(\pmin-1,\qmin+2.7)}{(\pmin+1.4,\qmin+2.4)}
\draw (\pmin-2.7,\qmin-0.5)   node[above] {\cite[Proposition 15]{CL25}};
\draw (\pmin-2.7,\qmin+1)   node[above] {Lim--McCann};
\draw (\pmin-2.7,\qmin+0.2)   node[above] {\cite[Cor.\,2.2]{LM21}};
\myarrowL{(\pmin-1,\qmin+1)}{(\pmin-0.08,\qmin+1.1)}
\def\sx{(\pmin-2.7} 
\def\sy{-1.3} 
\draw[simplexcolor!30, very thick, fill=simplexcolor!5] (\sx,\sy) -- (\sx-0.7,\sy-1) -- (\sx+0.3,\sy-1.4)  -- (\sx+0.55,\sy-0.7) -- cycle;
\draw[simplexcolor!30, very thick] (\sx,\sy) -- (\sx+0.3,\sy-1.4);
\draw[simplexcolor!30, dash pattern=on 3pt off 3pt, very thick] (\sx-0.7,\sy-1)  -- (\sx+0.55,\sy-0.7);
\filldraw[simplexpicturecolor] (\sx,\sy) circle (0.4ex);
\filldraw[simplexpicturecolor] (\sx-0.7,\sy-1)  circle (0.4ex);
\filldraw[simplexpicturecolor] (\sx+0.55,\sy-0.7) circle (0.4ex);
\filldraw[simplexpicturecolor] (\sx+0.3,\sy-1.4) circle (0.4ex);


\draw[black, thick]  (\qmin,\pmin) to (\n,\n);

\draw[->] (\pmin,0) to (\pmax,0);
\draw (\pmax,0)  node[right] {$p$};
\draw (\pmin-0.05,-0.1)  node[below] {$p=-\infty$};
\draw[black,fill=black] (\pmin-0.03,-0.2)  rectangle  (\pmin+0.03,0.2) ;

\draw[->] (0,\qmin) to (0,\qmax);
\draw (0,\qmax)  node[above] {$q$};
\draw (0,\qmin-0.05)  node[below] {$q=-\infty$};
\draw[black,fill=black] (-0.2,\qmin-0.03)  rectangle  (0.2,\qmin+0.03) ;

\draw[-] (\n,-0.1) to (\n,0.1);
\draw (\n+0.2, -0.1) node[below] {$n$};
\draw[-] (\n-2,-0.1) to (\n-2,0.1);
\draw (\n-2, -0.1) node[below] {$n-2$};
\draw[-] (-2,-0.1) to (-2,0.1);
\draw (-2, -0.1) node[below] {$-2$};

\draw[-] (-0.1,-2) to (0.1,-2);
\draw (-0.1,-2) node[left] {$-2$};
\draw[-] (-0.1,\n-2) to (0.1,\n-2);
\draw (-0.1,\n-2-0.25) node[left] {$n-2$};
\draw[-] (-0.1,\n-1) to (0.1,\n-1);
\draw (-0.1,\n-1) node[left] {$n-1$};
\draw[-] (-0.1,\n) to (0.1,\n);
\draw (-0.1,\n+0.2) node[left] {$n$};

\end{tikzpicture}
}
\caption{\label{fig:pqdiagram3D} ($n \geq 3$) Maximizing $\capq(K)/\capp(K)$ for $K \subset \Rn$. Solid regions of the diagram and the  solid segments indicate rigorous results. Striped regions and dashed curves are conjectural. Blue corresponds to the ball, red to the regular $(n+1)$-point set, and green to the two-point set. Symmetry breaking is shown for a certain region in the third quadrant in \autoref{chp:3d}. As the figure indicates, one should also expect symmetry breaking in part of the second quadrant. \textsc{Credit:} Figure adapted from Clark and Laugesen \cite[Figure 3]{CL25}.}
\end{figure}


In higher dimensions, we focus our interests on the region where $p<-2$ and $-2<q<0$.

We first show that like in two dimensions, a regular simplex in $\bbR^n$ surpasses the ball below a certain curve \autoref{eq:curve} in the parameter plane. We find that the curve gets closer to the lines $p=-2$ and $q=0$ when $n$ increases to infinity in \autoref{lemma:3dto-2} and \autoref{lemma:3dto0}.

In \autoref{sec:3d-algo}, we give explicit five-point and six-point configurations on the unit sphere whose capacity ratios exceed that of the ball: a split tetrahedron and a regular pentagonal pyramid. We compare their symmetry-breaking regions with that of the regular simplex using the curves on which each configuration's capacity ratio equals that of the ball. As in two dimensions, these numerical curves exhibit an intersection phenomenon: the regular-simplex curve crosses the curves for both configurations in \autoref{fig:capratio_comparison_higher_dim}(b). On the sampled grid, the regions are not all nested, and using configurations with more points does not yield a uniformly larger symmetry-breaking region.

\section{ \bf Notation}\label{chp:9}

As a simplification, we introduce a new notation of capacity for negative parameters. We will keep using the notations in the following sections.

\begin{definition*}
	Given a compact set $K$ and $r>0$. We define a new notation for the energy when $p=-r$, by letting
	\begin{align*}
		U_{r}(K) \coloneqq V_{-r}(K) = \max_{\mu(K)=1} \int_K \int_K |x-y|^r \,d\mu(x)\,d\mu(y),
	\end{align*}
	and we define
	\begin{align*}
		C_{r}(K) = \operatorname{Cap}_{-r}(K),
	\end{align*}
	so that
	\begin{align*}
		C_{r}(K) = U_{r}(K)^{1/r} = 
		\left( \max_{\mu(K)=1} \int_K \int_K |x-y|^r \,d\mu(x)\,d\mu(y)\right)^{\! \!1/r}.
	\end{align*}
	 Note that when $K$ is a finite set with $N$ points, the energy can be written as
	 \begin{align}\label{eq:discrete_energy}
	 	U_{r}(K) \coloneqq \max_{\sum_{i=1}^N m_i=1} \sum_{i\neq j}m_i m_j|x_i - x_j|^r,
	 \end{align}
	 where $m_i \ge 0$ for all $i$.
\end{definition*}

\section{\bf One dimension}\label{chp:1d}

This section gives numerical support for the conjectured optimal shapes in \autoref{fig:pqdiagram1D} when $p$ and $q$ are in $[-1, 0)$. To help computation, we also develop a positivity theorem which states that the equilibrium masses are positive at each point.

\subsection{Conjectures}

The two main conjectures we study in one dimension are:

\begin{conjecture}[Interval conjecture, see Clark and Laugesen \protect{\cite[Conjecture 12]{CL25}}]\label{thm:1d-interval}
	Let $0<r<s<1$ and $I=[-1,1]$ be the closed interval.  Then 
	\begin{align*}
		\min_{K} \frac{\caps(K)}{\capr(K)} = \frac{\caps(I)}{\capr(I)},
	\end{align*}
	where $K$ ranges over compact subsets of $\bbR$ containing at least two points.
\end{conjecture}

The right side in the conjecture is computable because the capacity of the interval $I = [-1, 1]$ is given by
	\begin{align*}
		\capr(I) = \left(\frac{\Gamma((1+r)/2) \Gamma(1-r/2)}{\Gamma(1/2)}\right)^{\!\!1/r}, \quad 0<r<1.
	\end{align*}
	The formula can be found in \cite[Appendix A]{CL24b} with $p=-r$.
By scaling invariance we know that any closed bounded interval $[a,b]$ has the same capacity ratio.

\begin{conjecture}[Two-point set conjecture, see Clark and Laugesen \protect{\cite[Conjecture 16]{CL25}}]\label{thm:1d-2pt}
	Let $0<r<s<1$. Then 
	\begin{align*}
		\max_{K} \frac{\caps(K)}{\capr(K)} = \frac{\caps(\{0,1\})}{\capr(\{0,1\})}.
	\end{align*}
\end{conjecture}

The right side in this conjecture is also computable because the capacity of the set $\{0, 1\}$ with two points is
	\begin{align*}
		\capr(\{0,1\}) = \frac{1}{2^{1/r}}, \quad r > 0,
	\end{align*}
and the equilibrium measure has half the mass at each of the two points, which one can check directly as follows: If $\mu$ has mass $a$ at $0$ and mass $b$ at $1$, where $a+b=1$, then $\int_K \int_K |x-y|^r d\mu d\mu = 2ab = 2a(1-a) \le 1/2$, with equality when $a=b=1/2$.


\subsection{Discrete sets and the Positivity Theorem}

Before developing numerical evidence for the conjectures, we want to understand the equilibrium measure better, by proving that the mass at every point in a discrete set is positive. The theorem also greatly helps improve the later algorithm to compute the capacity, in \autoref{sec:1d-numerical}.

Why do we care about discrete sets? Because by \cite[Proposition 19]{CL25} we know that when $r>0$, the capacity of any compact set can be approximated by a sequence of discrete subsets of $K$, and it is feasible to compute the capacity numerically for discrete sets, in the following section.

\begin{theorem}[Positivity Theorem for discrete sets]\label{thm:positivity}
	Let $0<r<1$ and $N\ge 1$. If $K=\{x_1, \ldots, x_N \} \subset \bbR$ is a finite set of distinct points, then the optimal weights $m_i$ for the energy in \eqref{eq:discrete_energy} are positive.
\end{theorem}

\begin{proof}
	When $N=1$, the only mass is $m_1 = 1$.
	
	Suppose $N \ge 2$ and consider a measure that is supported on $k < N$ points, which we suppose are $x_1, \ldots, x_k$, so that $m_i > 0$ for $i=1, \ldots, k$, and $m_1 + \cdots + m_k = 1$. We will show that this measure is not a maximizer, which implies that the equilibrium measure (optimal weights) must have $k=N$.
	
		We prove that the appearance of $x_{k+1}$ will increase the energy, no matter whether $x_{k+1}$ is on one side of the previous $k$ points, or in between some $x_i$ and $x_j$.

		Case (a). Suppose $x_{k+1} < x_i$ for $i=1,\ldots, k$ and $x_j$ is the closest point to $x_{k+1}$. Let the new weights be $m_i' = m_i$ for $i\neq j$ and $m_j' = m_j - \epsilon$, $m_{k+1}'= \epsilon$.
			Consider the energy difference 
			\begin{align*}
			f(\epsilon) &= \frac{1}{2} \left( \sum_{i,l=1}^{k+1}m_i' m_l' |x_i-x_l|^r - \sum_{i,l=1}^k m_i m_l |x_i - x_l|^r \right) \\
			&= \sum_{i=1}^k m_i(m_j-\epsilon) |x_i-x_j|^r + \sum_{i=1}^{k} m_i \epsilon |x_i-x_{k+1}|^r \\
			&\quad - \epsilon^2|x_j-x_{k+1}|^r - \sum_{i=1}^k 	m_i m_j |x_i-x_j|^r.
			\end{align*}
			Notice that $f(0) = 0$ and
			\begin{align*}
				f'(\epsilon) = \sum_{i=1}^k m_i \left( |x_i - x_{k+1}|^r - |x_i - x_j|^r \right) - 2\epsilon |x_j - x_{k+1}|^r,
			\end{align*}
			so that
			\begin{align*}
				f'(0) = \sum_{i=1}^k m_i \left( |x_i - x_{k+1}|^r - |x_i - x_j|^r \right).
			\end{align*}
			Notice that $|x_i-x_{k+1}| >|x_i-x_j|$, since $x_{k+1} < x_j < x_i$ for all $i=1,\ldots, k$ and $i\neq j$. Also, $x_i - x_j = 0$ when $i=j$. Hence $f'(0)>0$, meaning that by putting a small positive weight on $x_{k+1}$ the energy of the measure increases.
			
		Case (b). Suppose $x_{k+1} > x_i$ for $i=1,\ldots, k$. The proof is similar to $(a)$.
		
		Case (c). Suppose $x_{k+1}$ is in between some $x_i$ and $x_j$. Pick the points closest to $x_{k+1}$ such that $\ldots < x_i < x_{k+1} < x_j < \ldots$. Choose $\alpha \in (0,1)$ such that $x_{k+1} = \alpha x_i + (1-\alpha) x_j$. Let the new weights be $m_{k+1}' = \epsilon, m_i' = m_i - \alpha \epsilon, m_j' = m_j - (1-\alpha)\epsilon$ and $m_l' = m_l$ for all other indices $l$.
		

			Consider the energy difference
			\begin{align*}
				f(\epsilon) = \frac{1}{2} \left( \sum_{n,l=1}^{k+1}m_n' m_l' |x_n-x_l|^r - \sum_{n,l=1}^k m_n m_l |x_n - x_l|^r \right).
			\end{align*}
			Notice that $f(0)=0$ and
			\begin{align*}
				f'(0) = \sum_{l=1}^k m_l \left( |x_l - x_{k+1}|^r - \alpha |x_l - x_i|^r - (1-\alpha) |x_l - x_j|^r \right).
			\end{align*}
			By plugging in $x_{k+1} = \alpha x_i + (1-\alpha) x_j$, we have $|x_{l} - x_{k+1}| = |\alpha (x_l-x_i) + (1-\alpha) (x_l - x_j)|$. First notice that $x_l - x_i$ and $x_l - x_j$ have the same sign. By strict concavity of the function $g(u) = u^r$ (using that $0<r<1$), we have $|x_{l} - x_{k+1}|^r > \alpha |x_l-x_i|^r + (1-\alpha) |x_l - x_j|^r$. Therefore, $f'(0)>0$.
\end{proof}

\subsection{Algorithm}\label{1dim-algo}

Due to the Positivity Theorem \ref{thm:positivity}, we are able to use Lagrangian multipliers directly to compute the capacity. Specifically, with the Positivity Theorem we are sure that all the masses are nonzero, and so we can solve the masses directly from the formula of Lagrange multipliers.

Consider the Lagrange function 
\begin{align*}
	\cL(m_1, \ldots, m_N, \lambda) = \sum_{i\neq j}m_i m_j|x_i - x_j|^r - \lambda \left(\sum_{i=1}^N m_i - 1 \right).
\end{align*}
Since the optimal masses $m_i$ are known to be positive, they must satisfy the Lagrange multiplier condition $\nabla_{m, \lambda} \, \cL (m_1, \ldots, m_N, \lambda)=0$. That is
\begin{align*}
	\frac{\partial \cL}{\partial m_i}(m_1, \ldots, m_N, \lambda) &= 2 \sum_{j\neq i} m_j|x_i-x_j|^r - \lambda=0, \\
	\frac{\partial \cL}{\partial \lambda} (m_1, \ldots, m_N, \lambda) &= 1-\sum_{i=1}^N m_i = 0.
\end{align*}
We can rewrite the equations above in matrix form:
\begin{align*}
	\begin{bmatrix}
	0 & 2|x_1-x_2|^r & \cdots & 2|x_1 - x_N|^r & -1 \\
	2|x_2-x_1|^r & 0 & \cdots & 2|x_2 - x_N|^r & -1 \\
	\vdots & \vdots & \ddots & \vdots & -1 \\
	2|x_N-x_1|^r & \cdots & \cdots & 0 & -1 \\
	1 & 1 & 1 & 1 & 0
\end{bmatrix}
\begin{bmatrix}
m_1 \\ \vdots \\ m_N \\ \lambda
\end{bmatrix}
 = \begin{bmatrix}
0 \\ \vdots \\ 0 \\ 1
\end{bmatrix}.
\end{align*}

The augmented matrix is nonsingular: otherwise a nonzero mass-preserving null direction would keep the quadratic energy constant from a positive maximizer to a boundary point of the probability simplex, contradicting \autoref{thm:positivity}. The masses $m_1, \ldots, m_N$ of the equilibrium measure are obtained by multiplying both sides of the equation by the inverse matrix.

\subsection{Numerical results for discrete sets}\label{sec:1d-numerical}

To provide numerical support for \autoref{thm:1d-interval} and \autoref{thm:1d-2pt}, we sample numerous $N$-point sets and display the resulting histogram in \autoref{fig:1dim:hist}. All sampled ratios lie between the interval and two-point ratios. As $N$ increases, the sampled minima tend to decrease towards the interval ratio.
 
\begin{figure}[!htbp]
\centering
\setlength{\captionindent}{0pt}
\includegraphics[width=\textwidth]{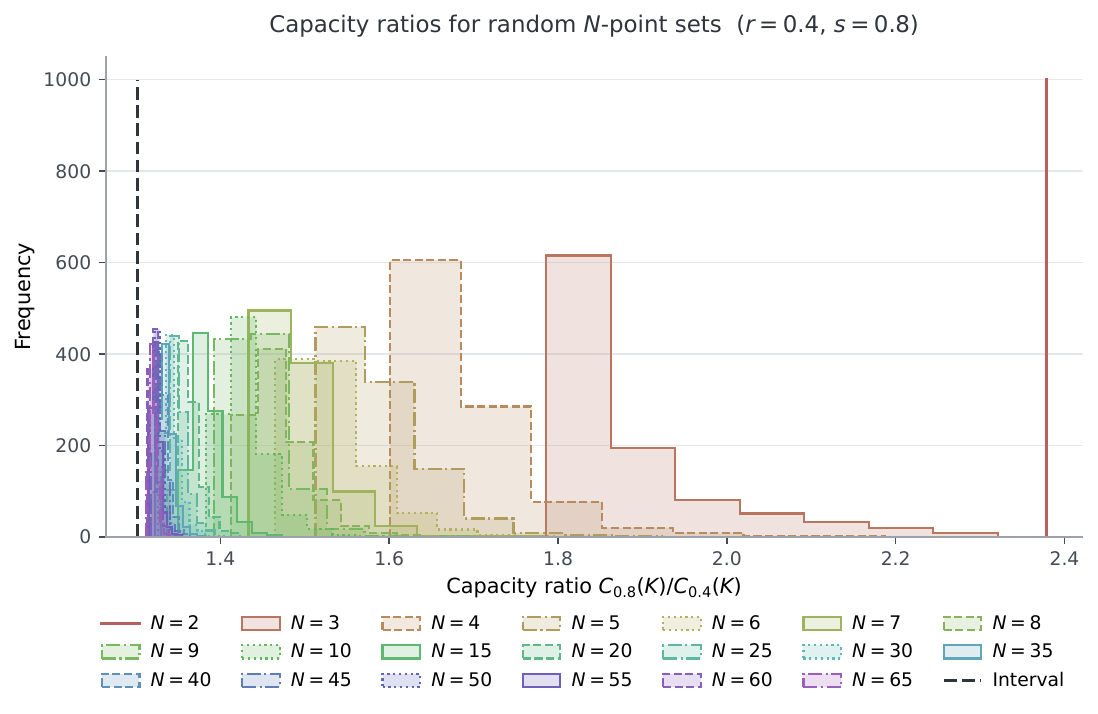}
	\caption{Capacity-ratio histograms for $r=0.4$ and $s=0.8$, using $1000$ random $N$-point sets in the unit interval for each $N=2,3,\ldots,10,15,20,\ldots,65$. Capacities are computed using \autoref{1dim-algo}. As $N$ increases, colors progress from red to purple through the spectrum and the distributions shift left, indicating smaller ratios on average. The reference lines mark the two-point ratio $\approx 2.38$ (red) and interval ratio $\approx 1.30$ (black). The code is available on \href{https://github.com/vhdvhd/Riesz-capacity-ratios-with-negative-exponents}{GitHub}.}
	\label{fig:1dim:hist}
\end{figure}

\FloatBarrier 
\subsection{Approximation of the equilibrium distribution}

We also compare the discrete equilibrium distribution for 1000 points with the interval equilibrium distribution, whose density on $(-1,1)$ for $r=-p \in (0,1)$ is (see \cite[Appendix A]{CL24b})
\begin{align*}
	d \mu (x) = \frac{\Gamma(1-r/2)}{\Gamma(1/2) \Gamma((1-r)/2)}\frac{dx}{(1-x^2)^{(1+r)/2}}.
\end{align*}

The comparison of the cumulative distribution functions of 1000 points and the interval $[0,1]$ is in \autoref{fig:1dim:interval}. The close agreement away from the endpoints gives confidence in our numerical algorithm.

\begin{figure}
\centering
\setlength{\captionindent}{0pt}
\includegraphics[width=0.95\textwidth]{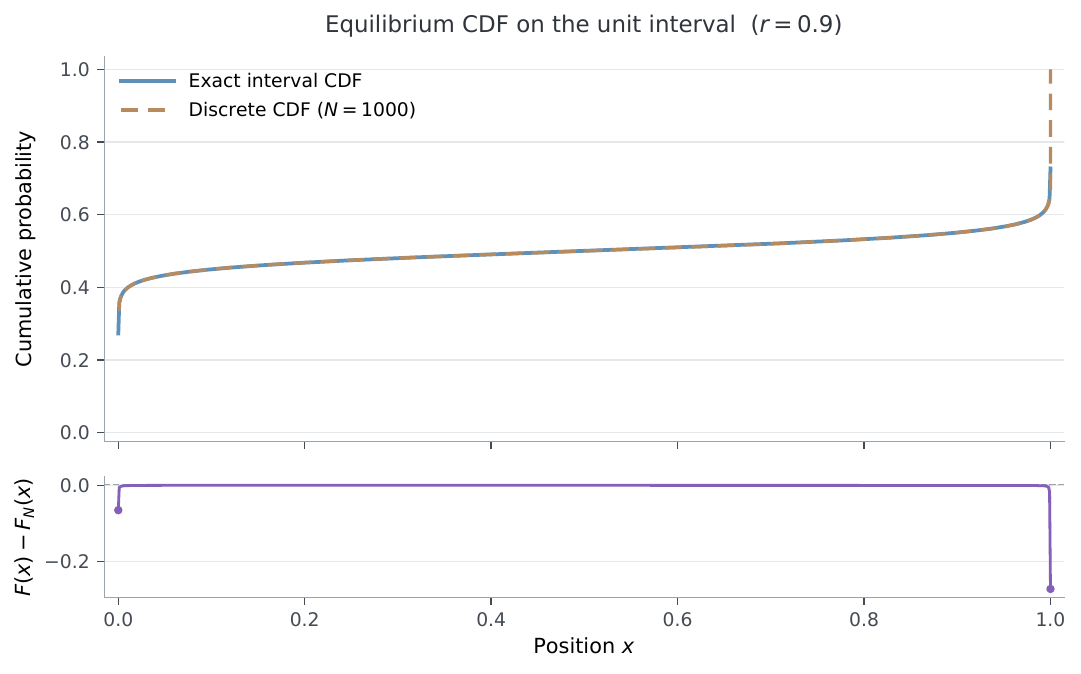}
	\caption{For $r=0.9$, we compare the theoretical interval CDF $F$ (blue) with the numerical CDF $F_N$ (orange) of the energy-maximizing mass distribution on $1000$ equally spaced points in $[\epsilon,1-\epsilon]$, where $\epsilon=5\times10^{-6}$. The lower panel shows $F-F_N$ evaluated at the support points, joined by line segments; agreement is close away from the endpoints. The code is available on \href{https://github.com/vhdvhd/Riesz-capacity-ratios-with-negative-exponents}{GitHub}.}
	\label{fig:1dim:interval}
\end{figure}

\FloatBarrier 
We place the 1000 points in $[\epsilon,1-\epsilon]$ so that the numerical CDF vanishes at $0$. Endpoint discrepancies remain; with fixed $\epsilon>0$, this does not give uniform convergence to the CDF on $[0,1]$ as the number of points increases.

Our method chooses 1000 equally spaced points and computes the optimal mass distribution for these points. Alternatively, one could find the optimal point configuration in $[0,1]$ with equal masses at each point. In that case, the numerical CDF converges to the interval CDF as the number of points increases to infinity \cite[Theorem 4.2.2]{BHS19}.

\section{ \bf Two dimensions}\label{chp:2d}

We investigate which set maximizes the capacity ratio when $r\ge 2, 0< s <2$ (the unknown region $p\le -2, -2< q <0$ in \autoref{fig:pqdiagram2D}).

The capacity of the unit disk $\overline{{\bbB}^2}$ is given in \cite[Appendix A]{CL24b}:
\begin{equation}\label{eq:capacity_B2}
	\capr(\overline{{\bbB}^2})
	= \left\{
	\begin{aligned}
		&2^{1-1/r} \quad & \text{when } r\ge 2,\\
		&2 \left( \frac{\Gamma((1+r)/2)}{\sqrt{\pi} \,\Gamma(1+r/2)} \right)^{\! 1/r} & \text{when } 0<r<2.
	\end{aligned}
	\right.
\end{equation}

Therefore, the capacity ratio of the closed unit disk when $r\ge 2, 0< s <2$, is
\begin{equation*}
	\frac{C_s(\overline{{\bbB}^2})}{C_r(\overline{{\bbB}^2})} = 2^{1/r} \left( \frac{\Gamma((1+s)/2)}{\sqrt{\pi} \,\Gamma(1+s/2)} \right)^{\! 1/s}.
\end{equation*}

Our goal is to find shapes with higher capacity ratios than the disk, in parts of the region  $r\ge 2, 0< s <2$.

\subsection{Capacity of regular $N$-gon}

We discuss the capacity of regular $N$-gons and the vertex set of the regular $N$-gon. 

\begin{definition*}
	Let $Q_N$ be the regular $N$-gon, and $P_N$ be the vertex set of $Q_N$, both of which are inscribed in the unit circle ($N\ge 3$). Notice that $Q_N$ is the convex hull of $P_N$.
\end{definition*}

\subsection*{Why consider the vertex set $P_N$ of a regular $N$-gon?}

The capacity ratio problem considers all the sets in $\bbR^2$. Why do we concentrate here on regular $N$-gons? Specifically, why do we only consider the vertex set?

The reason is that we want to find shapes that can beat the unit disk, which is symmetric under any angle rotation. The first example we can think of is the family of regular polygons, which are symmetric under certain angle rotations. From the following lemma, we only need to consider the boundary of $Q_N$.

\begin{lemma}[Bj\"{o}rck \protect{\cite[Theorem 3]{B56}}]
	When $r>0$, the support of any equilibrium measure of a compact set $K \subset \bbR^n$ ($n\ge 2$) is a subset of the boundary $\partial K$.
\end{lemma}

 When $r>2$, Bj\"{o}rck \cite[Theorem 12]{B56} further showed that the equilibrium measure is supported on at most $n+1$ extreme points of the convex hull. For regular $N$-gons in $\bbR^2$, the extreme points are the vertices and so:

\begin{lemma}\label{lm:3points}
	When $r>2$, the capacity of the vertex set $P_N$ and the regular $N$-gon $Q_N$ are equal, and any equilibrium measure is supported on at most $3$ vertices.
\end{lemma}

It is numerically and theoretically tractable to compute capacity for finite sets when $r>0$ to get symmetry-breaking examples, as we do below for $P_N$. Notice that $Q_N$ must have greater or equal capacity ratio as its vertex set $P_N$, because they have the same capacity in the denominator ($r> 2$) by \autoref{lm:3points} but in the numerator ($0<s<2$), the capacity of $Q_N$ is greater than or equal to the capacity of its subset $P_N$. Therefore if the capacity ratio for $P_N$ exceeds that of the ball, then the same holds for $Q_N$.


\subsection*{Calculating capacity of $P_N$}
We first notice that when $0<r<2$, the capacity of $P_N$ is easily calculated by symmetry.

\begin{lemma}\label{lemma:equal}
	If $0<r<2$, then $P_N$ has energy $U_r(P_N) = \frac{1}{N} \sum_{k=1}^{N-1}|e^{2k\pi i/N}- 1|^r$, and the equilibrium measure is equally distributed on every vertex of $P_N$.
\end{lemma}

\begin{proof}
	The uniqueness of the equilibrium measure (see \cite[Theorem 4]{B56}, \cite[Theorem 4.4.8]{BHS19}) and the rotational invariance of $P_N$ imply that each point in $P_N$ carries mass $1/N$. The energy formula then follows easily. 
\end{proof}

When $r \ge 2$, we need to separately discuss the cases when $N$ is odd or even. The following theorem gives us more information of the case when $N$ is even. For even $N$ and $r>2$, it implies that every equilibrium measure of $P_N$ is supported on a pair of antipodal points.

\begin{theorem}[Building on \protect{\cite[Theorem 4.6.6]{BHS19}}]\label{thm:r=2equilibrium}
	Let $K = \{x_i\}_{i=1}^\ell$ be a finite subset of the closed unit disk. Let $\ell_1 = \# \{K \cap \bbS^1\}$ be the number of the $x_i$ lying on the unit circle. Assume $\ell_1 \ge 2$.
	
	Suppose $r=2$ and the origin is contained in the convex hull of $K \cap \bbS^1$. Then the energy $U_2(K) = 2$, and any measure supported on $K \cap \bbS^1$ with center of mass at the origin is an equilibrium measure. When $\ell_1 = 2$, each point in $K \cap \bbS^1$ has half the mass and the points are antipodal. When $\ell_1 = 3$, the equilibrium measure is unique. When $\ell_1>3$ and the origin lies in the interior of this convex hull, the equilibrium measures form an $(\ell_1-3)$-dimensional family.
		
	Suppose $r>2$ and $K\cap \bbS^1$ contains two antipodal points. Then $U_r(K) = 2^{r-1}$, and the measure with half its mass at each antipodal point is an equilibrium measure, and every equilibrium measure has that form.
\end{theorem}

\begin{proof}
Consider a measure $\mu$ on $K$ with $\mu(\{x_i\})=m_i \ge 0$. Let $b=(b_1, b_2)$ be the center of mass, i.e. $b = \sum_{i=1}^\ell m_i x_i.$ Then 
\begin{align*}
	\sum_{i=1}^\ell \sum_{j=1}^\ell m_i m_j x_i \cdot x_j 
	= b \cdot b
	= |b|^2.
\end{align*}
Hence when $r=2$, 
\begin{align}\label{eq:bound}
	\sum_{i=1}^\ell \sum_{j=1}^\ell m_i m_j |x_i - x_j|^2 = 2 \sum_{i=1}^\ell m_i |x_i|^2 - 2|b|^2 \le 2-2|b|^2 \le 2,
\end{align}
	where $``="$ holds if and only if $b=0$ and $\sum_{i=1}^\ell m_i |x_i|^2=1$. These conditions are satisfied for some $\mu$, since the origin is contained in the convex hull of the points in $K$ with $|x_i|=1$. Thus $U_2(K)=2$.
	
	We now aim to characterize all equilibrium measures when $r=2$. The maximal energy is attained if and only if the masses are all supported at points on the unit circle and the center of mass $b=0$. For these constraints, discard the interior points and relabel so that $\ell=\ell_1$. Writing $x_i = (x_{i,1}, x_{i,2})$, these constraints say
	\begin{align*}
		\sum_{i=1}^\ell m_i x_{i,1} = 0, \quad \sum_{i=1}^\ell m_i x_{i,2} = 0,\quad \sum_{i=1}^\ell m_i = 1
	\end{align*}
	forming a $3 \times (\ell+1)$ augmented matrix 
	\begin{align*}
	 \setlength{\arraycolsep}{6pt}{
		A = \begin{bmatrix}
			x_{1,1} & \cdots  & x_{\ell,1} & \vline & 0  \\
			x_{1,2} & \cdots & x_{\ell,2} & \vline & 0   \\
			1 & \cdots & 1 & \vline & 1 
		\end{bmatrix}.
		}
	\end{align*}
	
	When the masses are supported on $\ell=2$ points, $A$ has rank two because the first two rows are linearly dependent from $m_1 x_1 + m_2 x_2 = 0$. The third row is surely independent of them since the bottom right entry is nonzero. Since $x_1$ and $x_2$ are on the unit circle, the two points must be antipodal and the masses are $m_1 = m_2 = 1/2$. 
	
	When $\ell\ge 3$, the first two rows are linearly independent because a line meets the unit circle at at most two points. The third row is independent of them since the bottom right entry is nonzero. Thus the affine solution space has dimension $\ell-3$. Its intersection with the nonnegative orthant is precisely the family of equilibrium measures. Under the interior hypothesis this intersection contains a strictly positive vector, so it has the same dimension.
		
	Returning to all points of $K$, when $r>2$,
\begin{align*}
	\sum_{i=1}^\ell \sum_{j=1}^\ell m_i m_j |x_i - x_j|^r 
	&= \sum_{i=1}^\ell \sum_{j=1}^\ell m_i m_j |x_i - x_j|^{r-2} |x_i - x_j|^2 \\
	&\le 2^{r-2} \sum_{i=1}^\ell \sum_{j=1}^\ell m_i m_j |x_i - x_j|^2\\
	& \le 2^{r-1} \quad \text{by } \eqref{eq:bound}.
\end{align*}
	
We observe that when $K\cap\, \bbS^1$ contains two antipodal points $a$ and $-a$ (which have distance $|a-(-a)|=2$), equality holds when $m_a = 1/2, m_{-a} = 1/2$. It suffices to prove that no other masses attain equality. If ``='' holds for some other measure, then $\sum_{i=1}^\ell \sum_{j=1}^\ell m_i m_j |x_i - x_j|^2=2$ and $|x_i - x_j|=2$ whenever $i\ne j$ and $m_i$ and $m_j$ are positive. If there were three or more points on the unit circle with positive masses, then there would exist $x_i, x_j$ such that $|x_i - x_j| < 2$ and the energy would be less than $2^{r-1}$. Thus exactly two masses are positive and those points must be antipodal.
\end{proof}

 When $N$ is even and $r>2$, the equilibrium measure is supported on two antipodal points by \autoref{thm:r=2equilibrium}, and hence the capacity is easy to compute, as follows.
%

\begin{corollary}[$N$ even]
\label{cor:n_even}
	Let $P_N$ be the vertex set of a regular $N$-gon inscribed in the unit circle. Suppose $N$ is even. If $r=2$, then the energy $U_2(P_N) = 2$, and any measure on $P_N$ with center of mass at the origin is an equilibrium measure. If $r > 2$, then the energy $U_r(P_N) = 2^{r-1}$, and the equilibrium measure consists of two equal masses on diametral points of $P_N$
\end{corollary}

Recall that Bj\"orck \cite[Theorem 12]{B56} with $n=2$ says that when $r>2$, the equilibrium measure of a compact set $K\subset \bbR^2$ is supported on at most $3$ points. Thus another possible proof for \autoref{cor:n_even} is that if we already knew that the three points of $P_N$ (at which the equilibrium measure is supported) include two antipodal points, then  the triangle with vertices at the three points would be a right triangle. Therefore, $c^2 = a^2 + b^2$ where the triangle has side lengths $a\le b < c$. Since $r>2$, 
	\begin{align*}
		\left(\frac{a}{c}\right)^{\! r} + \left(\frac{b}{c}\right)^{\! r} < \left(\frac{a}{c}\right)^{\! 2} +\left(\frac{b}{c}\right)^{\! 2} = 1.
	\end{align*}
	Hence by \autoref{lemma:3point} below, we would know the masses are constrained on the two antipodal vertices.

	When $N$ is odd, we raise:
\begin{conjecture}[$N$ odd]\label{conj:n-gon}
	Let $r>2$ and $P_N$ be the vertex set of a regular $N$-gon inscribed in the unit circle ($N\ge 3$). Suppose $N$ is odd. Let $\theta_N = \frac{(N-1)\pi }{2N}$. Then the energy equals
	\begin{align*}
		U_r(P_N) = \frac{2^{r-1}\sin^r \theta_N}{1-2^{r-2}\cos^r \theta_N},
	\end{align*}
	and an equilibrium measure is supported on the three points $1, e^{2\theta_N i}, e^{2(\pi-\theta_N)i}$ (up to rotation), with masses $$\frac{2b^r - a^r}{4b^r - a^r}, \quad  \frac{b^r}{4b^r - a^r} , \quad \frac{b^r}{4b^r - a^r}$$
	respectively, where $a = 2\sin(\pi/N)$ and $b=c=2\sin\theta_N$.
	The formulas for the energy are explained in \autoref{sec:n-gon}.
\end{conjecture}

The conjectured energy and proposed maximizing measure agree with exhaustive triple comparisons on the plotting grids for odd $N\le 50$, using \autoref{lemma:3point} below. These finite checks do not establish uniqueness up to rotation. We have a partial proof of the conjecture, which we now describe.

\subsection{Our partially successful attempt to prove \autoref{conj:n-gon}}
\label{sec:n-gon}

In this section we attempt to prove that when $r>2$, an equilibrium measure of $P_N$ is supported on the three vertices given in \autoref{conj:n-gon}. We decompose the attempted proof into two steps. In the first step we handle the symmetric case when one point is fixed and the other two points are symmetric. This step is proved. In the second step we attempt to reduce the general case to the symmetric case. This step remains incomplete. 

Since the maximal masses of $P_N$ are supported on at most three points when $r>2$, we can use the following lemma to compute the energy for three points:

\begin{lemma}[Capacity of a three-point set; Clark and Laugesen \protect{\cite[Theorem 8]{CL25}}]\label{lemma:3point}
	Let $r>0$. If $T\subset \bbR^2$ is a three-point set with distances $a,b,c$ between its points, where $0 < a,b \le c$, then $T$ has energy
	\begin{equation*}
		U_r(T) = 
		\left\{
		\begin{array}{lr}
			 2^{-1}c^{r} & \text{if} \ a^r + b^r \le c^r, \\
			 \frac{2a^rb^rc^r}{4(ab)^r - (a^r+b^r-c^r)^2}  &\text{if} \ a^r + b^r > c^r.
		\end{array}\right.
	\end{equation*}
\end{lemma}

In the first case, the equilibrium measure is supported at the two points that are distance $c$ apart. In the second case, the equilibrium measure is supported at all three points.

\subsection{Step 1 --- fully proved}

The following lemma shows that for $r>2$, the energy increases as the symmetric points move towards $(-1,0)$ within the stated angular range. Thus, in this symmetric case, it is better to choose the two vertices closest to $(-1,0)$.

\begin{figure}
\centering
\includegraphics[totalheight=7cm]{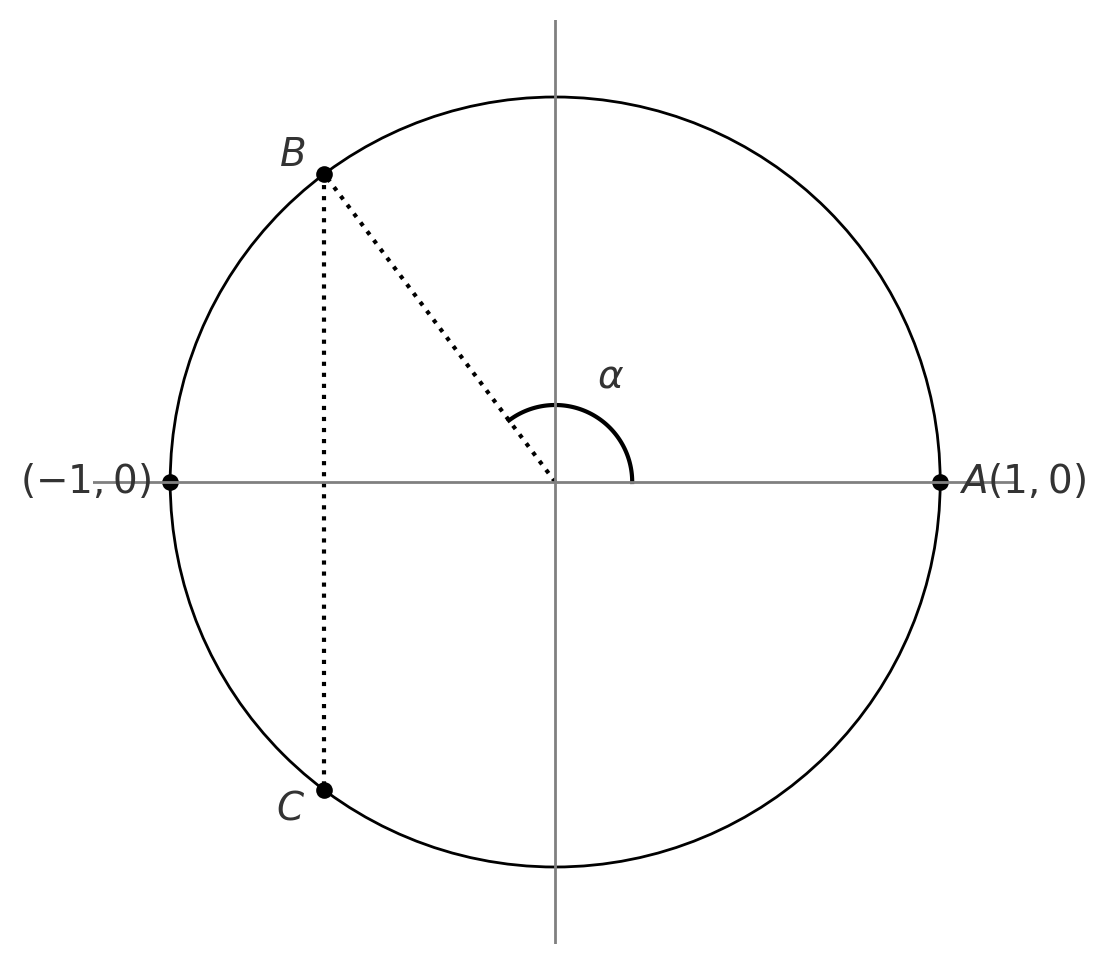}
	\caption{Fix $A=(1,0)$ and move $B$, $C$. Here $B$ and $C$ are symmetric across the horizontal axis. For $r\ge2$ and $2\pi/3\le\alpha\le\pi$, the capacity is nondecreasing as $B$ and $C$ move towards $(-1,0)$, and is strictly increasing for $r>2$.}
	\label{fig:3pt_1_}
\end{figure}

\begin{lemma}\label{lm:11.8}
	Suppose $r \ge 2$. Fix $A=(1,0)$ and let $B=e^{i\alpha}, C=e^{-i\alpha}$, where $\alpha\in[2\pi/3,\pi)$. The capacity of $\{A,B,C\}$ is nondecreasing in $\alpha$, and strictly increasing if $r>2$. The conclusion extends to $\alpha=\pi$ by continuity.
\end{lemma}

\begin{proof}
	For $2\pi/3\le\alpha<\pi$, $AB$ and $AC$ are longest sides of the triangle $ABC$. Moreover, we can compute
	\begin{align*}
		AB = AC = 2\sin (\alpha/2), \quad BC = 2\sin{\alpha}.
	\end{align*}
	By \autoref{lemma:3point}, with $a=2\sin\alpha, b=c=2\sin{(\alpha/2)}$, we find the equilibrium measure is supported on all three points, and the energy is given by
	\begin{align*}
		U_r(\alpha) 
		= \frac{2^{r-1} \sin^r(\alpha/2)}{1 - 2^{r-2} \cos^r(\alpha/2)}.
	\end{align*}
	The derivative is
	\begin{align*}
		U_r'(\alpha) = \frac{2^{r-2}r \sin^{r-1}(\alpha/2) \cos(\alpha/2)}{\left(1-2^{r-2}\cos^r(\alpha/2)\right)^2}\left[1-\bigl(2\cos(\alpha/2)\bigr)^{r-2}\right].
	\end{align*}
	This is positive when $r>2$ and $2\pi/3<\alpha<\pi$. For $r=2$, the energy is constantly $2$. At $\alpha=\pi$, the formula extends to the two-point energy $2^{r-1}$.
\end{proof}

\subsection{Step 2 --- not proved}

In Step 2 we attempt to reduce a general triangle to an isosceles triangle with two equal longest sides, to which Step 1 applies.

\begin{figure}
\centering
\includegraphics[totalheight=7cm]{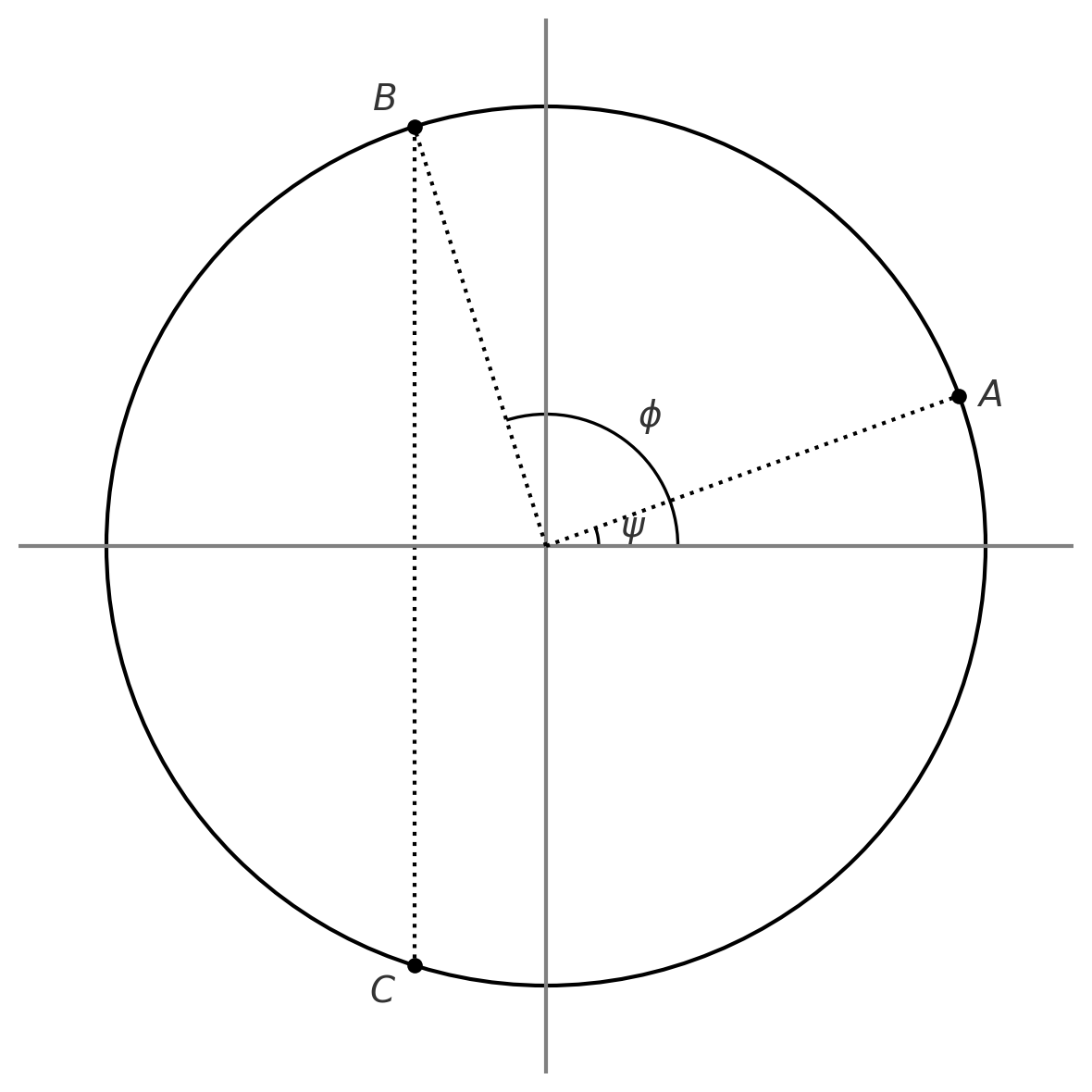}
	\caption{Fix $B, C$ to be symmetric across the horizontal axis, where $BC$ is required to be the longest edge in the triangle $\Delta ABC$. The angle of $B$ is $\phi \in (\pi/2, 2\pi/3]$. Move $A$ on the arc $\psi \in [0, 2\pi - 3\phi]$. For $r\ge 2$, our conjecture is that when $\psi = 2\pi - 3\phi$, the $r$-capacity of the 3-point set $\{A,B,C\}$ is maximal.}
	\label{fig:3pt_1}
\end{figure}

\begin{conjecture}\label{conj:bc_a}
Suppose $r \ge 2$. Let $A, B, C$ be three points on the unit circle. 
Fix $B$ and $C$ such that they are symmetric with respect to the $x$-axis, 
having polar angles $\phi$ and $-\phi$, respectively, where $\phi \in (\pi/2, 2\pi/3]$.
Move $A$ along the circle with polar angle $\psi \in [0, 2\pi - 3\phi]$, 
so that $BC$ remains the longest side of the triangle $\triangle ABC$. 
Then the capacity of the three-point set $\{A, B, C\}$ 
is maximized when $\psi = 2\pi - 3\phi$, i.e., when $AC = BC$.\label{conj:3pt}
\end{conjecture}

\autoref{conj:n-gon} would follow from applying \autoref{conj:bc_a} and then \autoref{lm:11.8}: obtuse or right triangles reduce to the two-point case, while an acute vertex triangle can be placed in the Step 2 parametrization, whose maximizing endpoint is still a vertex triangle of $P_N$.

Some numerical evidence for \autoref{conj:bc_a} is below. In the plotted examples with $r>2$, the capacity is constant on the two-point branch and increasing on the three-point branch. In both cases, it is largest at the right endpoint of the $\psi$-interval.

\begin{figure}
  \centering
  \begin{minipage}[b]{0.49\textwidth}
    \vspace{0pt}
    \centering
    \includegraphics[width=\textwidth]{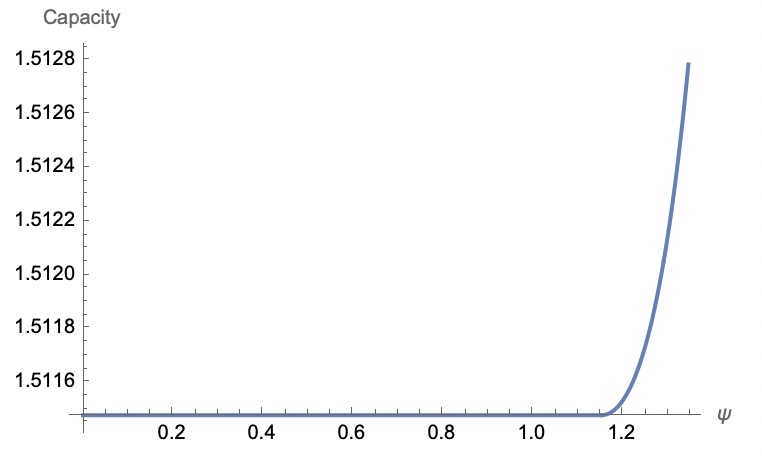}
    \caption{Capacity vs $\psi$, when $\phi=11\pi/21$ is close to $\pi/2$ and $r=2.5$.}
    \label{fig:abc1}
  \end{minipage}
  \hfill
  \begin{minipage}[b]{0.49\textwidth}
    \vspace{0pt}
    \centering
    \includegraphics[width=\textwidth]{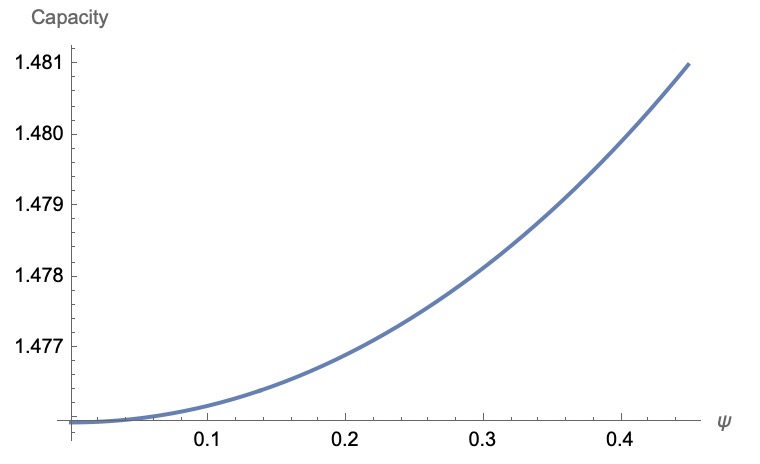}
    \caption{Capacity vs $\psi$, when $\phi=13\pi/21$ is close to $2\pi/3$ and $r=2.5$.}
    \label{fig:abc2}
  \end{minipage}
\end{figure}

\FloatBarrier 
\subsubsection{Unsuccessful attempt to prove \autoref{conj:3pt}}
We do not have definite conclusions, but if someone in future wants to work on this problem, we hope that the remainder of this section might be useful.

Our goal is to prove that the capacity in \autoref{conj:3pt} is a nondecreasing function in $\psi$. Notice that the capacity is a piecewise function.
A key question is determining when to use two points versus three points to compute the capacity of a 3-point set, i.e. which case of \autoref{lemma:3point} applies. Assume $BC$ is the longest sidelength in the triangle. Notice that if $AB^r + AC^r - BC^r \le 0$ then \autoref{lemma:3point} tells us to use two points and if $AB^r + AC^r - BC^r>0$ then to use three points.

\paragraph{Numerical observation} The switch quantity $AB^r + AC^r - BC^r$ need not be monotone in $\psi$. Numerically, once it becomes positive, it remains positive for all larger $\psi$ in the stated interval. At the right endpoint $AC=BC$, so the quantity equals $AB^r>0$.

\begin{figure}
  \centering
  \begin{minipage}[b]{0.49\textwidth}
    \centering
    \includegraphics[width=\textwidth]{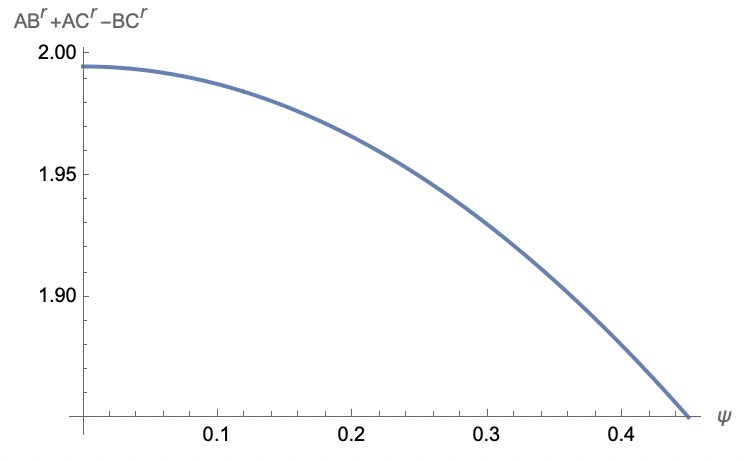}
    \caption{$AB^r + AC^r - BC^r$ vs $\psi$ when $\phi=13\pi/21$ and $r=2$.}
	\label{fig:abc_1}
  \end{minipage}
  \hfill
  \begin{minipage}[b]{0.49\textwidth}
    \centering
    \includegraphics[width=\textwidth]{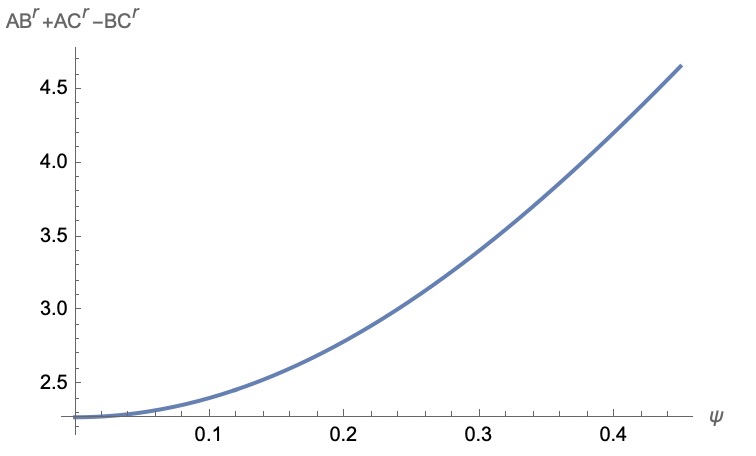}
    \caption{$AB^r + AC^r - BC^r$ vs $\psi$ when $\phi=13\pi/21$ and $r=5$.}
	\label{fig:abc_2}
  \end{minipage}
\end{figure}

\begin{figure}
  \centering
  \begin{minipage}[b]{0.49\textwidth}
    \centering
    \includegraphics[width=\textwidth]{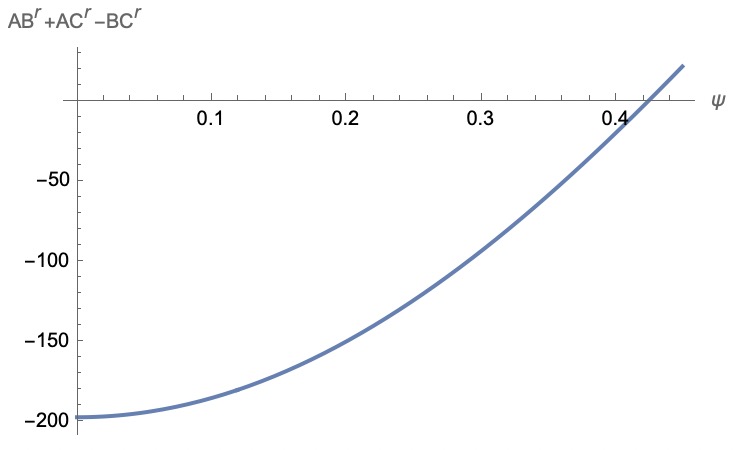}
    \caption{$AB^r + AC^r - BC^r$ vs $\psi$ when $\phi=13\pi/21$ and $r=10$.}
	\label{fig:abc_3}
  \end{minipage}
  \hfill
  \begin{minipage}[b]{0.49\textwidth}
    \centering
    \includegraphics[width=\textwidth]{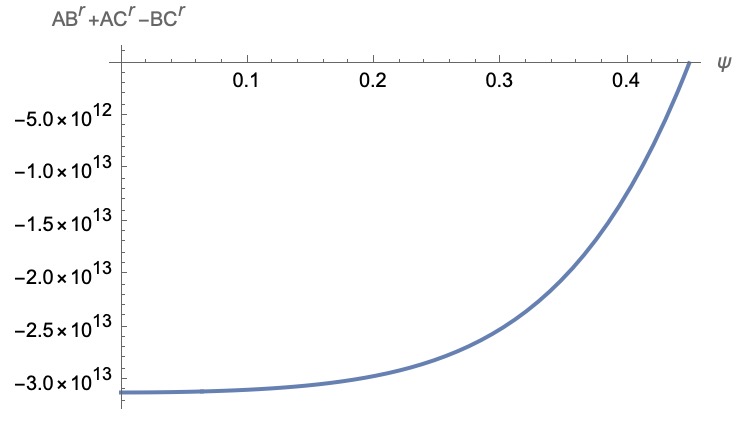}
    \caption{$AB^r + AC^r - BC^r$ vs $\psi$ when $\phi=13\pi/21$ and $r=50$.}
	\label{fig:abc_4}
  \end{minipage}
\end{figure}

\paragraph{Some computation} We first express the sidelengths in \autoref{fig:3pt_1} with $\phi$ and $\psi$:
$$BC = 2\sin\phi, \quad AB = 2\sin\frac{\phi - \psi}{2}, \quad AC = 2\sin \frac{\phi + \psi}{2}.$$

The quantity $AB^r + AC^r - BC^r$ is a function of $\psi$:
\begin{align*}
	l(\psi) = AB^r + AC^r - BC^r = 2^r \sin^r \frac{\phi - \psi}{2} + 2^r\sin^r \frac{\phi + \psi}{2} - 2^r\sin^r\phi.
\end{align*}

The derivative is
\begin{align*}
	l'(\psi) = 2^{r-1} r \left( \sin^{r-1}\frac{\phi + \psi}{2} \cos \frac{\phi+\psi}{2} - \sin^{r-1}\frac{\phi - \psi}{2} \cos \frac{\phi - \psi}{2}\right).
\end{align*}

Notice that $(\phi - \psi)/2 < (\phi + \psi)/2$ when $\psi >0,$ because
\begin{align*}
	\phi - \psi \in [4\phi - 2\pi, \phi] \subset [0,\phi],
	\quad
	\phi + \psi \in [\phi, 2\pi - 2\phi] \subset [\phi, \pi].
\end{align*}

Let $g(\theta) = \sin^{r-1}\theta \cos \theta$. The derivative can be rewritten as 
\begin{align*}
	l'(\psi) = 2^{r-1} r \left( g\left(\frac{\phi + \psi}{2}\right) - g\left(\frac{\phi - \psi}{2}\right) \right).
\end{align*}

On the relevant interval $(0,\pi/2)$,
\[
g'(\theta)
= \sin^{r-2}\theta\bigl(r\cos^2\theta-1\bigr).
\]
Thus $g'(\theta)=0$ at
$\theta^*=\arccos(1/\sqrt{r})$.
The function $g$ increases on $(0,\theta^*)$
and decreases on $(\theta^*,\pi/2)$.

\subsection{Capacity ratios}
Our ultimate objective is to identify the shape that maximizes the capacity ratio.
Thus far, we have analyzed the capacity of the regular polygon vertex set $P_N$. We now compare the capacity ratios of a disk and of $P_N$ for various parameter choices $s$ and $r$.

\subsection{Known results and conjectures}

Clark and Laugesen \cite[Theorem 10]{CL25} have solved the case $s\ge 2$ (equivalently $q\le -2$; see \autoref{fig:pqdiagram2D}), so we focus on $0<s<2$.
In the subregion $0<r<s<2$ (the lower triangle of the capacity--ratio in \autoref{fig:pqdiagram2D}), they conjectured that two-point sets are maximal.
Our subsequent numerical computations support this conjecture.

\begin{conjecture}[Two-point conjecture, see Clark and Laugesen \protect{\cite[Conjecture 16]{CL25}}]
	Suppose $n \ge 2$ and $0<r<s<2$. Among compact sets
$K \subset \mathbb{R}^n$ containing more than one point, any set with
exactly two points maximizes $\caps(K)/\capr(K)$.
\end{conjecture}

Notice that two-point sets are not the only potential maximizers. Other sets can be maximizers if they have the same underlying measure.

In the upper triangle of the figure, where $0<s<2$ and $r>2$, Clark and Laugesen found the symmetry--breaking phenomenon: in the yellow region of \autoref{fig:pqdiagram2D}, the regular three-point set has a larger capacity ratio than the ball. 

\begin{theorem}[Symmetry breaking, see Clark and Laugesen \protect{\cite[Theorem 7]{CL25}}]
Let $q_* \simeq -0.856$ satisfy $2\sqrt{\pi}\,\Gamma\!\left(1-\frac{q_*}{2}\right)
= 3\,\Gamma\!\left(\frac{1-q_*}{2}\right)$.
If
\[
-\infty < p <
\frac{ q \log(4/3) }{
  \log\!\left(
    \dfrac{2\sqrt{\pi}\,\Gamma\!\left(1-\frac{q}{2}\right)}
          {3\,\Gamma\!\left(\frac{1-q}{2}\right)}
  \right)
}
\quad\text{and}\quad
-2 < q < q_*,
\]
then the ratio $\operatorname{Cap}_q(K)/\operatorname{Cap}_p(K)$ is less for $K$ a ball than when $K=P_3$ is the vertex set of an equilateral triangle.
\end{theorem}

This theorem directly motivates our focus on vertex sets of regular polygons.

\subsection{Our results on $P_N$}
For a theoretical result, we compare the capacity ratios of the regular polygon vertex set $P_N$ and the disk when $r \ge 2$ and $0 < s < 2$, and $N$ is odd.

The capacity ratio of a disk $D$ is 
\begin{align}\label{eq:cap_disk}
	\frac{\caps(D)}{\capr(D)} = 2^{1/r} \left(\frac{\Gamma((1+s)/2)}{\sqrt{\pi} \Gamma(1+s/2)} \right)^{\! 1/s}.
\end{align}
by the formula at the beginning of the section.

When $0 < s < 2$, the equilibrium measure of $P_N$ is equally distributed on every vertex from \autoref{lemma:equal}. For $r>2$, we assume \autoref{conj:n-gon} identifies a maximizing triple; at $r=2$, we use the center-of-mass result in \autoref{thm:r=2equilibrium}. Hence the capacity ratio is 
\begin{align*}
	\frac{\caps(P_N)}{\capr(P_N)}
	&= \frac{ \left( \frac{1}{N} \sum_{k=1}^{N-1}|e^{2k\pi i/N}- 1|^s \right)^{\! 1/s}}{\left(\frac{2^{r-1}\sin^r \theta}{1-2^{r-2}\cos^r \theta}\right)^{\! 1/r}}\\
	&= \frac{2^{1/r+1/s} \left(1-2^{r-2}\cos^r \theta\right)^{\! 1/r} \left(\sum_{k=1}^{(N-1)/2} \sin^s(k \pi / N)\right)^{\! 1/s} }{N^{1/s} \sin\theta},\\
\end{align*}
where $\theta = \frac{(N-1)\pi }{2N} \ge \pi/3$.
Therefore, the capacity ratio of $P_N$ is at least as large as that of disk $D$ when
\begin{align*}
	\frac{2^{1/r+1/s} \left(1-2^{r-2}\cos^r \theta\right)^{\! 1/r} \left(\sum_{k=1}^{(N-1)/2} \sin^s(k \pi / N)\right)^{\! 1/s} }{N^{1/s} \sin\theta} \ge 2^{1/r} \left(\frac{\Gamma((1+s)/2)}{\sqrt{\pi} \Gamma(1+s/2)} \right)^{\! 1/s},
\end{align*} 
i.e., when
\begin{align*}
	f(r) = \left(1-2^{r-2}\cos^r \theta\right)^{\! 1/r} \ge \left( \frac{N\Gamma((1+s)/2)}{2\sqrt{\pi} \Gamma(1+s/2) \sum_{k=1}^{(N-1)/2} \sin^s(k \pi / N)} \right)^{\!\! 1/s}\sin\theta.
\end{align*}

We claim that the function $f(r) = \left(1-2^{r-2}\cos^r \theta\right)^{\! 1/r}$ is increasing. Notice $f(r) = p_3((p_2 \circ p_1)(r), r)$ where $p_1(r)=(2 \cos\theta)^r, p_2(x)=1-x/4, p_3(x,r)=x^{1/r}$. Clearly $p_1(r)$ is nonincreasing in $r$ because $0<2\cos\theta\le1$, and $p_2(x)$ is decreasing and $p_3(x,r)$ is increasing with respect to both $x\in (0,1)$ and $r>0$. Hence the composition $f(r)$ is increasing, with range $f([2,\infty))=[\sin\theta,1)$. When the argument below lies in this range, the polygon ratio is at least as large as the disk ratio precisely when
\begin{align} \label{ineq:region}
	r \ge f^{-1} \left( \left( \frac{N\Gamma((1+s)/2)}{2\sqrt{\pi} \Gamma(1+s/2) \sum_{k=1}^{(N-1)/2} \sin^s(k \pi / N)} \right)^{\! 1/s}\sin\theta \right).
\end{align}
If the argument is at least $1$, no finite $r\ge2$ gives a polygon ratio as large as the disk ratio. These comparisons remain conditional on \autoref{conj:n-gon} for $r>2$. We use the equality in this formula to create plots in \autoref{sec:intersection}.

\subsection{Numerical evidence}

\subsection{Algorithm}\label{2dim:algo} We compare the capacity ratios of the disk and regular polygon vertex sets $P_N$ among the plotted competitors. For either positive exponent, we use uniform masses below $2$ and the center-of-mass result at $2$. Above $2$, we use the antipodal formula for even $N$; for odd $N$, Bj\"orck's proved at-most-three-point support bound allows exhaustive evaluation of all vertex triples using \autoref{lemma:3point}. We use the closed formula of \autoref{conj:n-gon}, which specifies the maximizing triple, only after checking it numerically on the exact plotting grid. These finite checks do not prove the conjecture. Disk capacities are computed from the explicit formulas above.

\subsection{Observation when $0<r,s < 2$}\label{sec:2d0to2}

When $0<r,s < 2$, we notice that among the $P_N$ and the disk, the maximal shape $K$ for the capacity ratio $\caps(K)/\capr(K)$ is a disk when $r>s$ and a two-point set when $r<s$. See \autoref{fig:cr_2d_02}.

\begin{figure}
\centering
\includegraphics[totalheight=10cm]{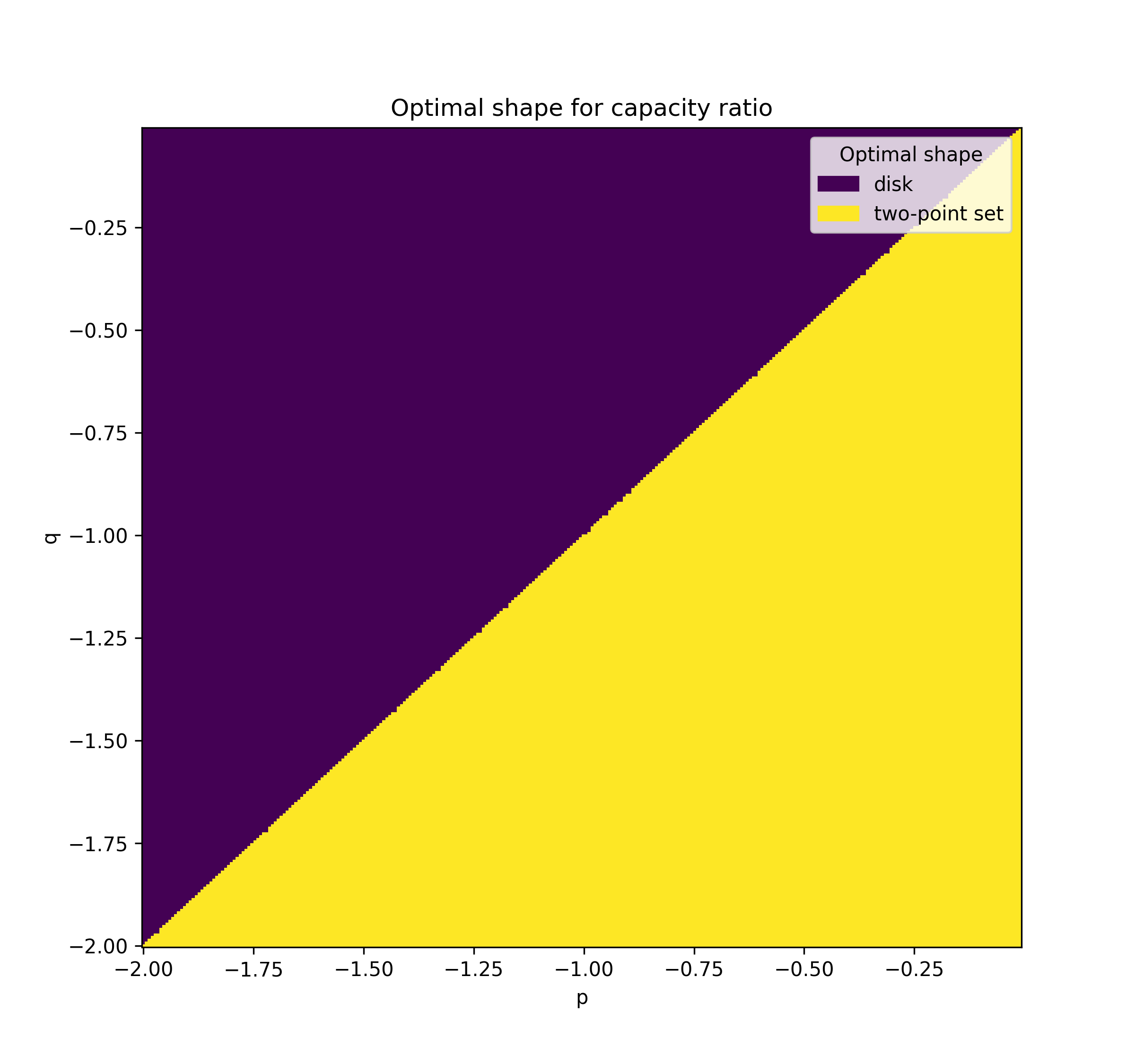}
	\caption{Optimal capacity ratio among the disk and regular polygon vertex sets $P_N$, $N=2,\ldots,50$, on a $300\times300$ uniform grid with $r,s\in[0.01,2]$. Here $p=-r, q=-s$.}
	\label{fig:cr_2d_02}
\end{figure}

\FloatBarrier 
\subsection{Observation when $r \ge 2$ and $0 < s < 2$}

We find that when $r \ge 2$ and $0 < s < 2$, the vertex sets $P_N$ with odd $N$ outperform the disk in some regions, with apparently nested phase boundaries partitioning the region into subdomains where each tested odd $N$ is optimal among the competitors. Beyond these boundaries (that is, for smaller values of $s$), the disk appears to maximize the capacity ratio among these competitors. See \autoref{fig:cr_2d}.

\begin{figure}
\centering
\includegraphics[totalheight=10cm]{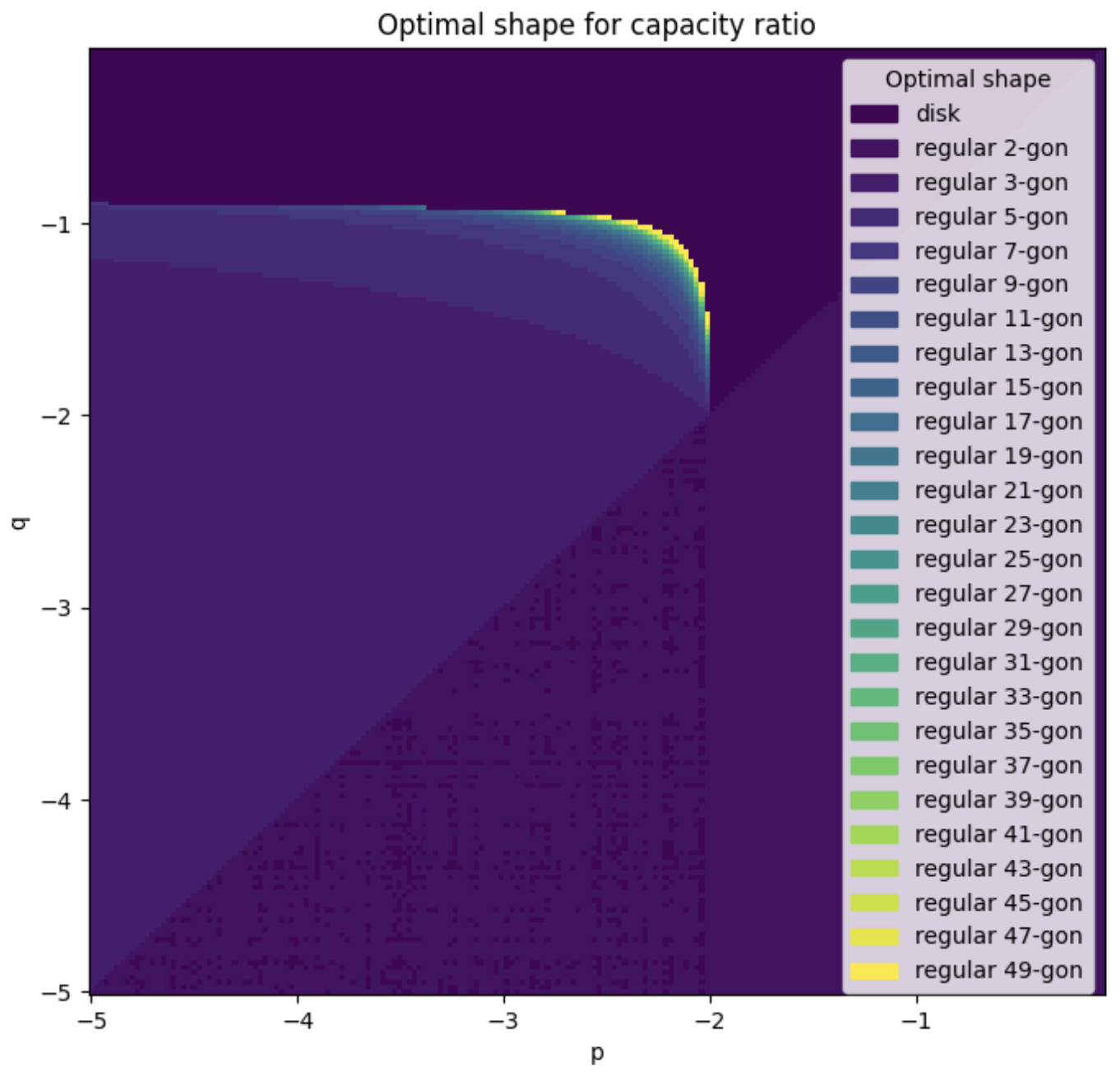}
	\caption{Optimal capacity ratio among the disk and regular polygon vertex sets $P_N$, $N=2,\ldots,50$, on a $198\times198$ uniform grid with $r,s\in[0.1,5]$. Displayed maximizers need not be unique. The algorithm is given in \autoref{2dim:algo}.} \label{fig:cr_2d}
\end{figure}

\FloatBarrier 
\subsection{Intersections}\label{sec:intersection}
\autoref{fig:cr_2d_1} and \autoref{fig:cr_2d_2} show winner-region maps for a smaller competitor family. The individual polygon--disk equality curves, plotted in Figures \ref{fig:capratio1}--\ref{fig:capratio3}, intersect rather than remain nested as $N$ increases.

\begin{figure}[htbp]
  \centering
  \begin{minipage}[t]{0.45\textwidth}
    \centering
    \includegraphics[width=\textwidth]{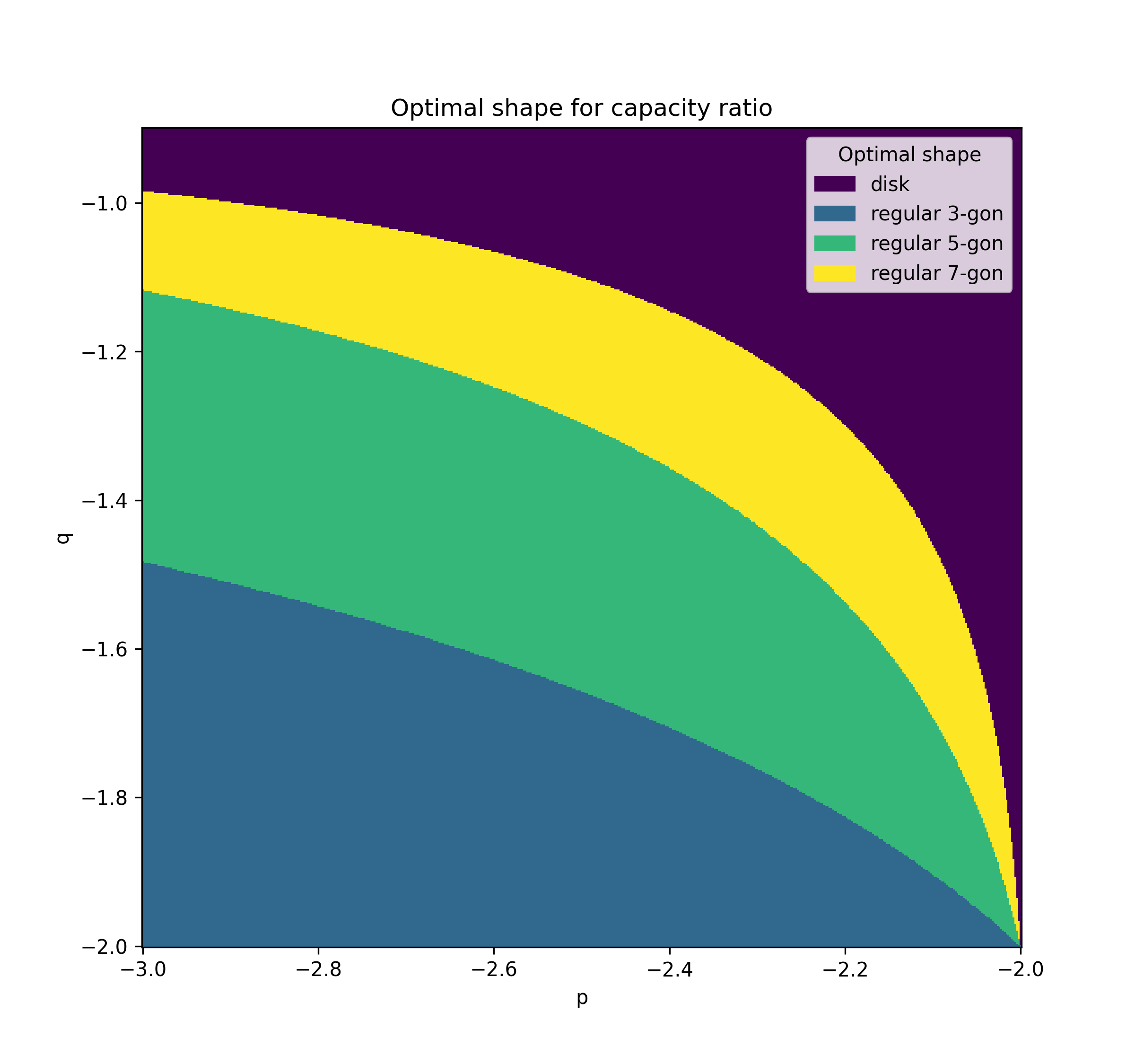}
    \caption{Best ratio among the disk and $P_N$, $N=2,\ldots,7$, near $(-2,-2)$. \label{fig:cr_2d_1}}
  \end{minipage}
  \hfill
  \begin{minipage}[t]{0.45\textwidth}
    \centering
    \includegraphics[width=\textwidth]{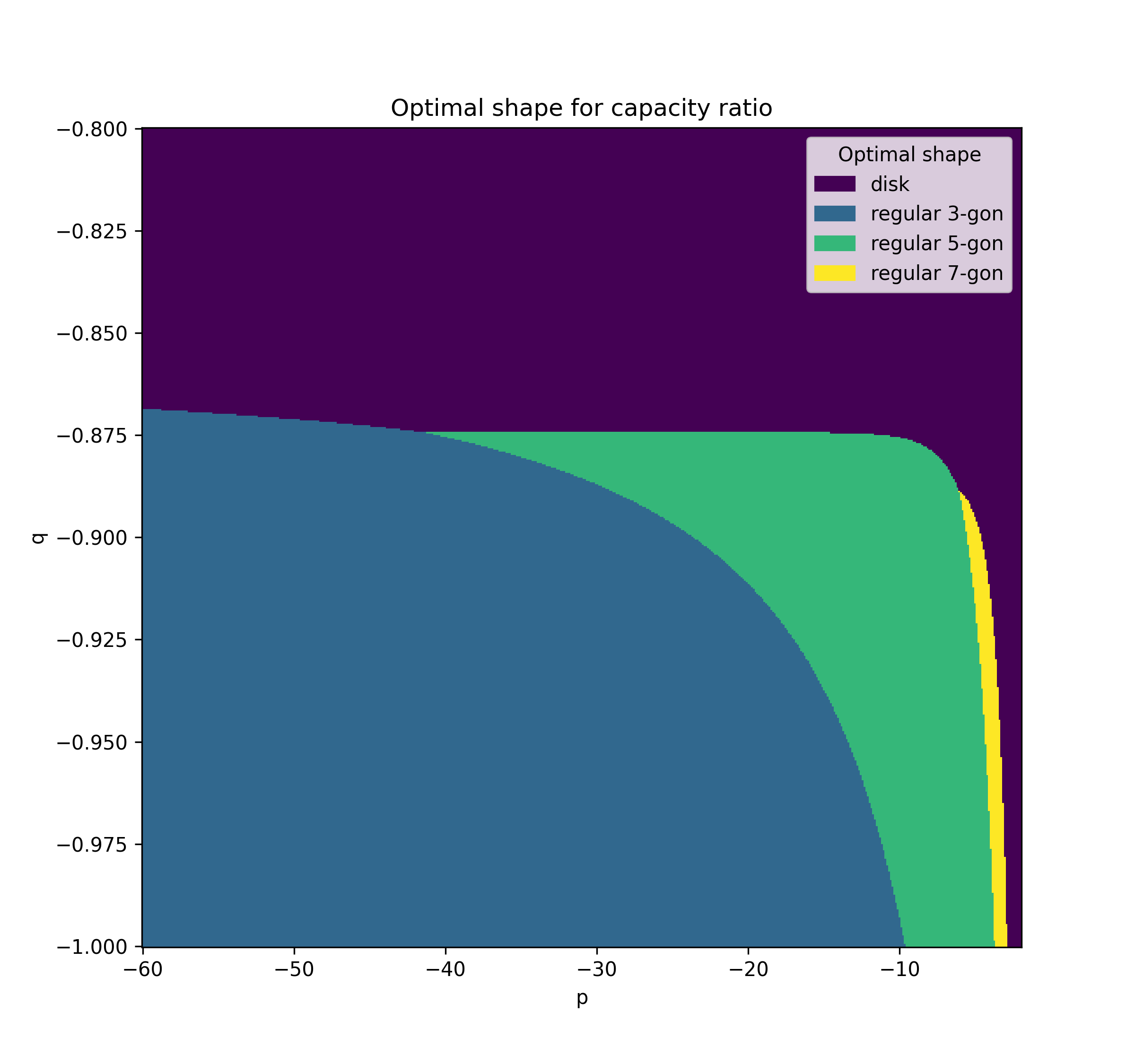}
    \caption{Best ratio among the disk and $P_N$, $N=2,\ldots,7$, zoomed-out view. \label{fig:cr_2d_2}}
  \end{minipage}
\end{figure}

\FloatBarrier 
We also plot the equality curves using \autoref{ineq:region}, whose vertex-set formulas are conjectural in general but verified numerically on the sampled grids for $N=3,5,7,9$ in Figures \ref{fig:capratio1}--\ref{fig:capratio3}. The region of $(p,q)=(-r,-s)$ under each curve in \autoref{fig:capratio1} is where the corresponding $P_N$ has a larger capacity ratio than the unit disk. At first glance, these disk-improvement regions might appear nested as $N$ increases, but a closer examination of \autoref{fig:capratio2} reveals intersections between the curves. All the curves originate from $(-2,-2)$.

\begin{figure}[htbp]
  \centering
  \begin{minipage}[b]{0.47\textwidth}
    \centering
    \includegraphics[width=\textwidth]{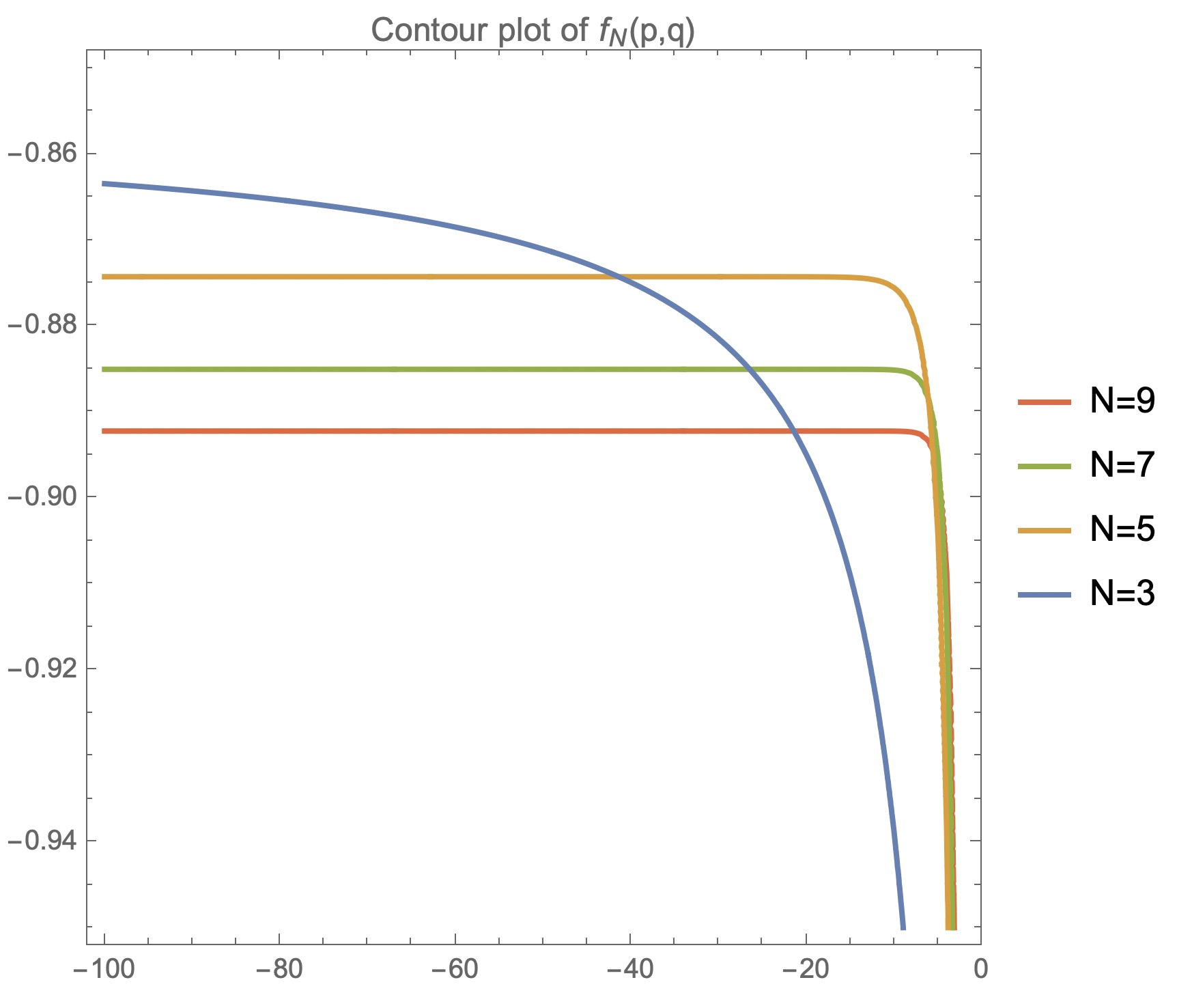}
    \caption{Zoomed-out view. \label{fig:capratio1}}
  \end{minipage}
  \hfill
  \begin{minipage}[b]{0.47\textwidth}
    \centering
    \includegraphics[width=\textwidth]{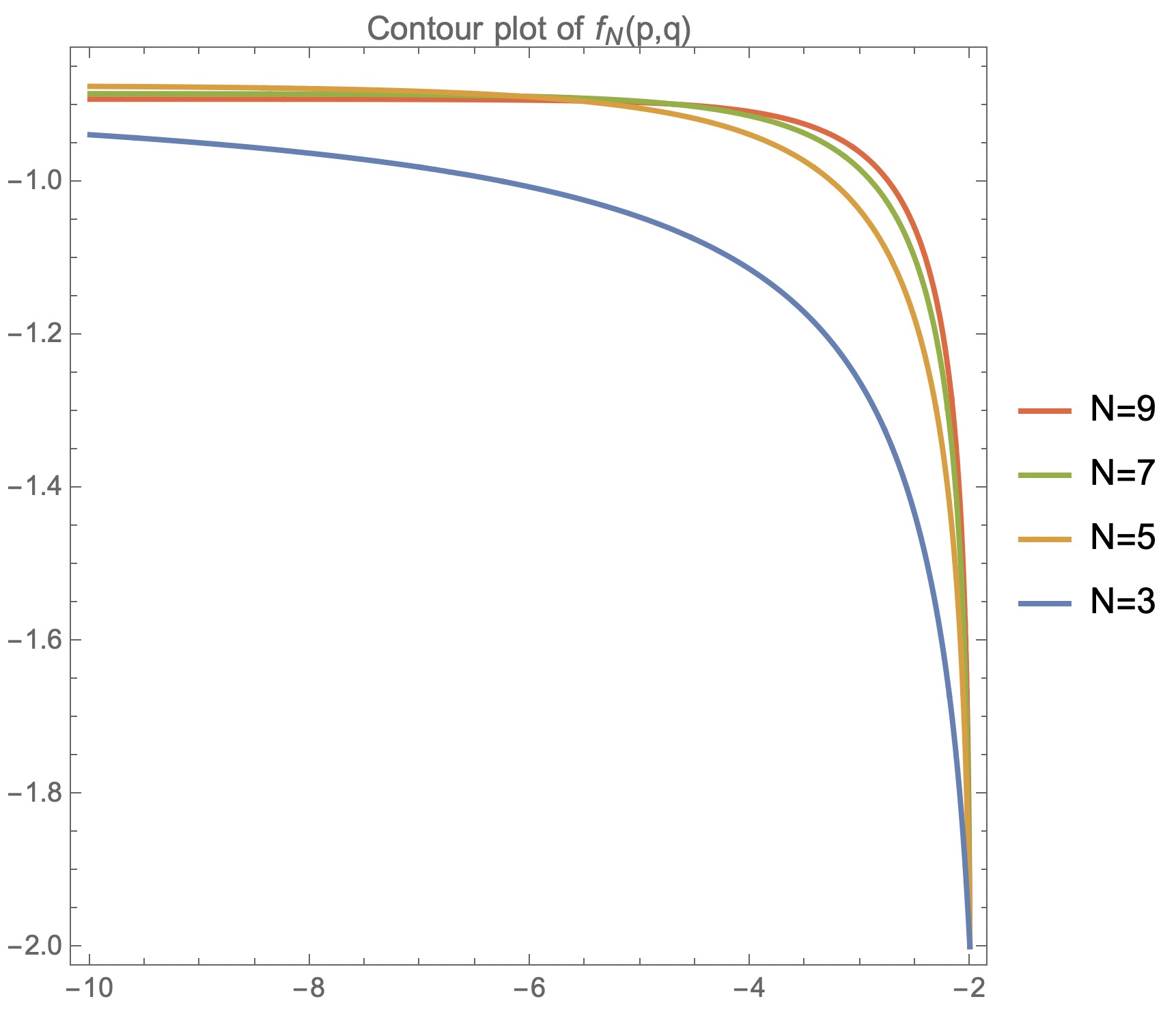}
    \caption{Zoomed-in view.\label{fig:capratio2}}
  \end{minipage}
  \begin{minipage}[b]{0.47\textwidth}
    \centering
    \includegraphics[width=\textwidth]{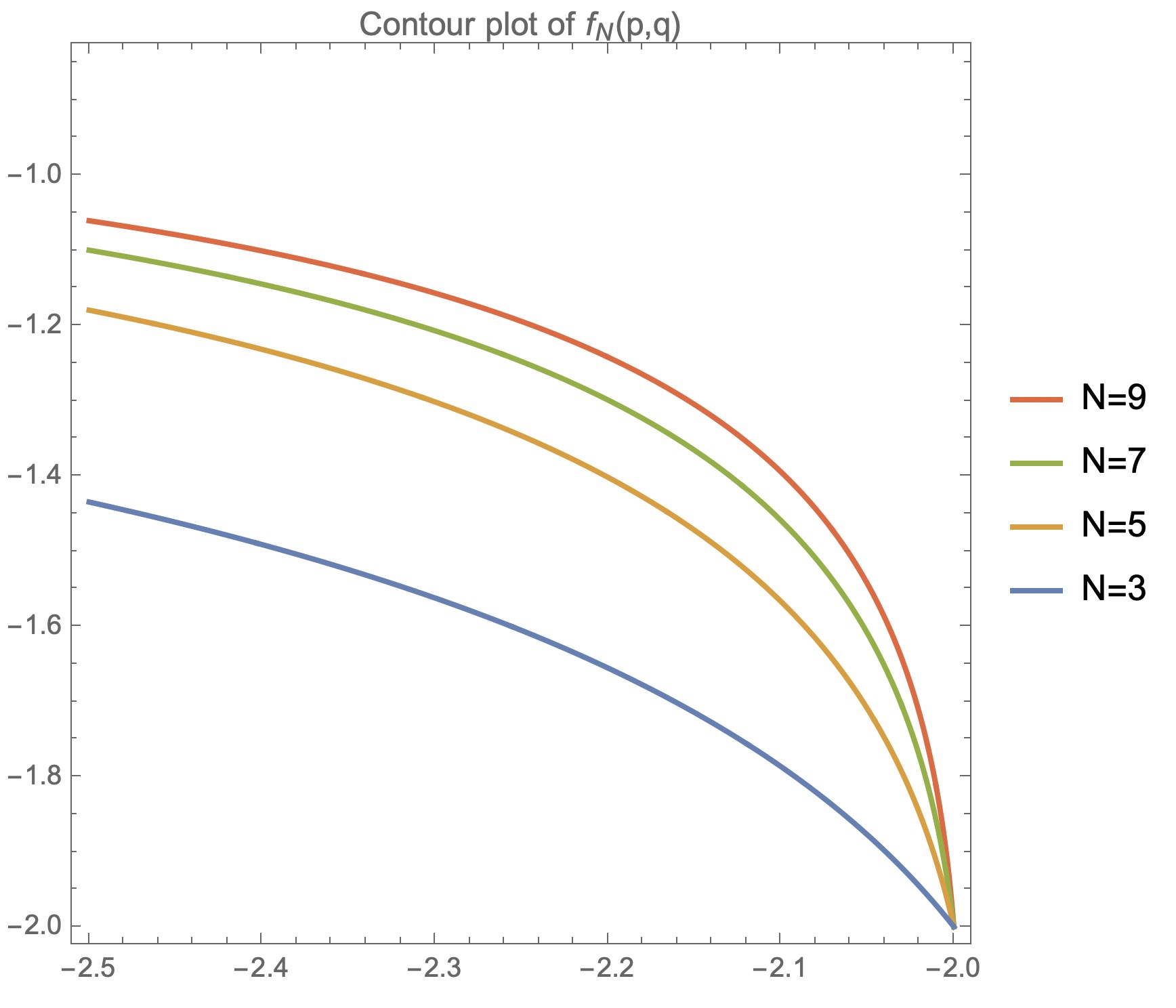}
    \caption{Near $(-2, -2)$. \label{fig:capratio3}}
  \end{minipage}
\end{figure}

\section{ \bf Higher dimensions}\label{chp:3d}

%
%
%
%

We investigate symmetry breaking for the capacity ratio conjecture in three and higher dimensions.

\subsection{Compare regular simplex and ball in higher dimensions}

We compare the regular $(n+1)$-point simplex $K_n \subset \bbR^n$ with diameter $1$ and the closed unit ball $\overline{\bbB^n}$. The capacity is known exactly for those sets, with
\begin{align}\label{eq:cap_simplex}
	\capr(K_n) =\left( \frac{n}{n+1} \right)^{\! 1/r}, \quad
	\capr(\overline{\bbB^n}) = 
	\begin{cases}
		2^{1-1/r}, & \text{if } r \ge 2,\\
		2 \left( \frac{\Gamma((n-1+r)/2) \Gamma(n-1)}{\Gamma(n-1+r/2) \Gamma((n-1)/2)}\right)^{\! 1/r}, & \text{if } 0 < r < 2.
	\end{cases}
\end{align}
For example, see \cite[Section 6]{CL25} and \cite[Appendix A]{CL24b}. The difference of capacity ratios $$\frac{\caps(K_n)}{\capr(K_n)} - \frac{\caps(\overline{\bbB^n})}{\capr(\overline{\bbB^n})}$$ when $n=3$ is plotted in \autoref{fig:capratio3dim1}, showing symmetry breaking like in two dimensions.

\begin{figure}
    \centering
    \includegraphics[width=0.6\textwidth]{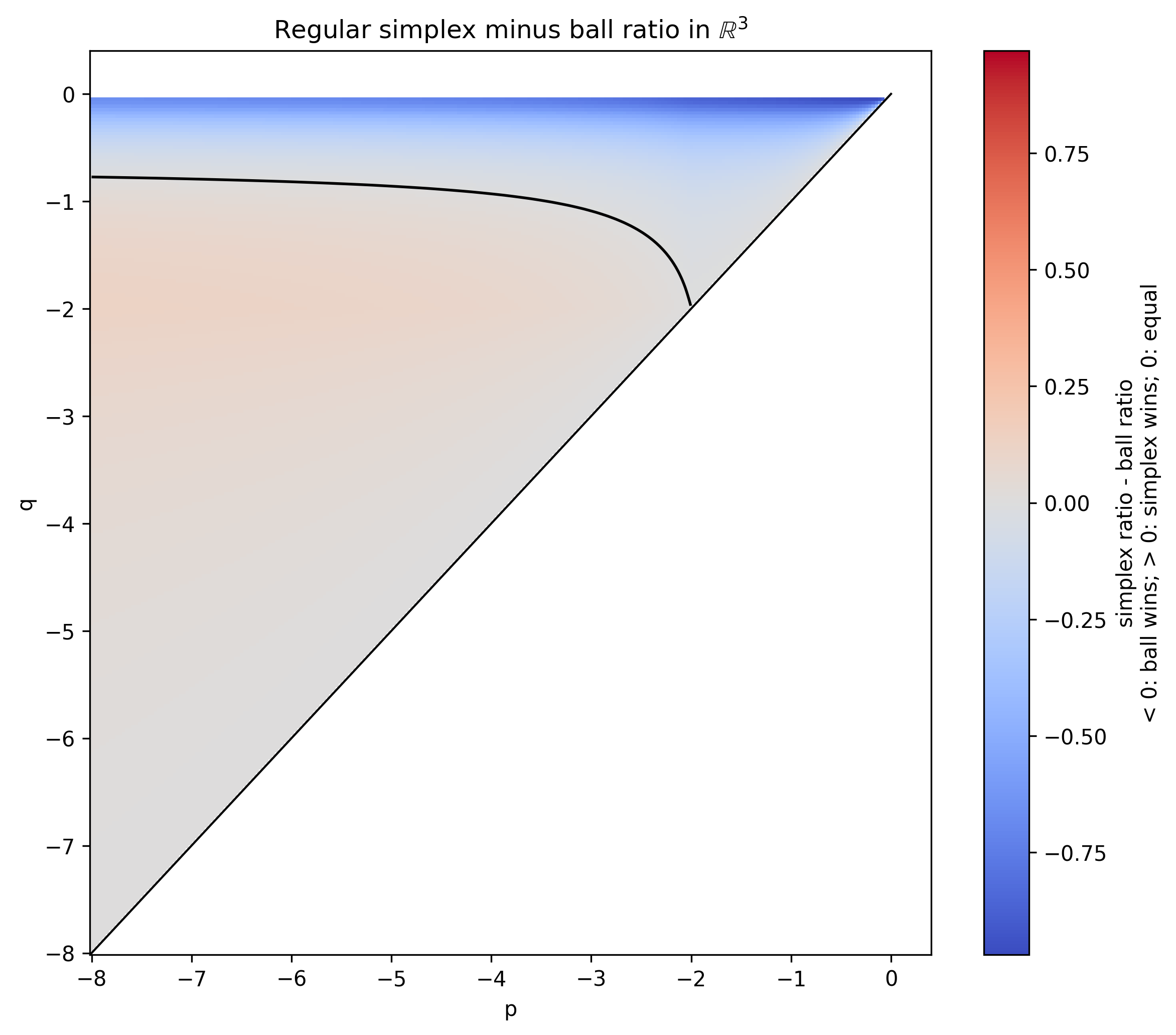}
    \caption{Simplex-minus-ball capacity-ratio difference for $n=3$, with the color scale centered at zero: positive (red) favors the simplex, and negative (blue) favors the ball. The black curve marks equality. As usual, $p=-r$ and $q=-s$. \label{fig:capratio3dim1}}
\end{figure}

\FloatBarrier 
\subsection{Understanding the high-dimensional limit}

Let us see how this symmetry breaking region given by the simplex changes as the dimension $n$ increases.

By equating the capacity ratios of the ball and regular simplex and substituting the capacity formula from \autoref{eq:cap_simplex}, we find a formula for $r$ in terms of $s$, and thus for $p$ in terms of $q$: 
\begin{align}\label{eq:curve}
	p = f(n, q) \coloneqq  \frac{q\log{\frac{2n}{n+1}}}{\log(n\Gamma(n-1-\frac{q}{2})\Gamma(\frac{n-1}{2}))-\log((n+1)\Gamma(\frac{n-1-q}{2})\Gamma(n-1))}.
\end{align}
In the region $p\le-2$, $-2<q<0$, this is an equality curve only for $-2<q<q^*(n)$, where the denominator is positive (see \autoref{lm:root_of_D}). In this range the simplex has the higher ratio exactly when $p<f(n,q)$. For $q^*(n)\le q<0$, the ball has the higher ratio for every finite $p\le-2$.

\begin{figure}
    \centering
    \includegraphics[width=0.64\textwidth]{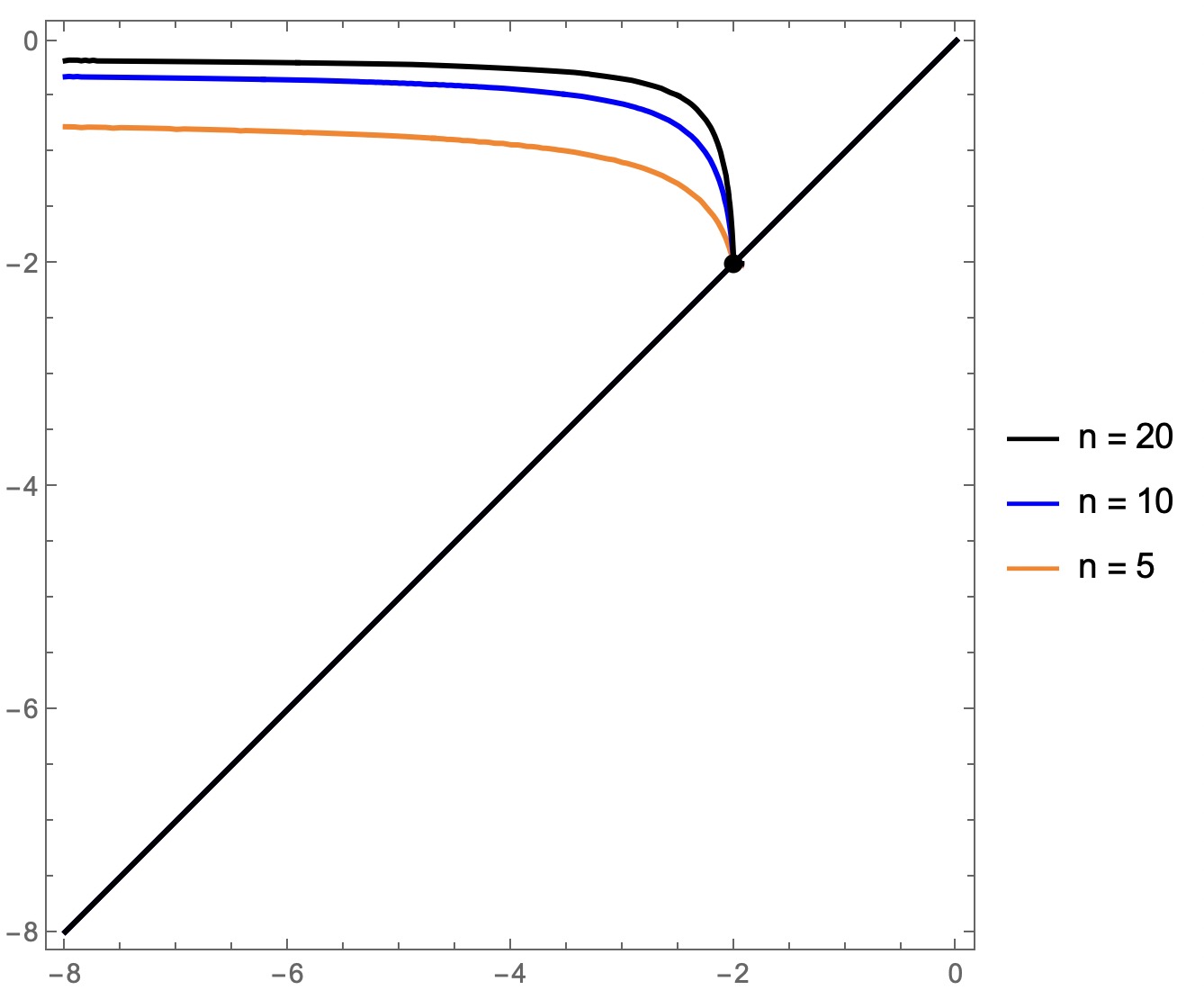}
    \caption{Compare the capacity ratio of regular simplex $K_n$ and ball $\overline{\bbB^n}$ in dimensions $n=5,10,20$. The curve indicates where the capacity ratios are equal. We observe that all curves pass the point $(p,q)=(-2, -2)$ and the curve gets closer to $q=0$ when $n$ increases. The horizontal asymptote as $p\to -\infty$ is $q^*(n)$. \label{fig:compare_n51020}}
\end{figure}

\FloatBarrier 
We want to understand properties of this curve.

\begin{lemma}\label{lm:root_of_D}
	For each fixed $n\ge2$, the denominator
	\begin{align}\label{eq:def_denom}
	D(n, q)
\coloneqq\log\bigl(n\,\Gamma(n-1-\tfrac q2)\,\Gamma(\tfrac{n-1}2)\bigr)
-\log\bigl((n+1)\,\Gamma(\tfrac{n-1-q}2)\,\Gamma(n-1)\bigr).
	\end{align}
	of $f(n,q)$ has a unique root $q^*(n) \in (-2, 0)$, and the denominator is positive when $-2<q<q^*(n)$.
\end{lemma}

\begin{proof}

\textbf{Existence:}
We apply the Intermediate Value Theorem by showing $D(n,0) < 0$ and $D(n,-2) > 0$:

\begin{itemize}
  \item If \(q=0\), then
  \begin{align}\label{eq:D(q,0)}
    D(n, 0)
    =\log\frac{n}{n+1}<0.
  \end{align}
%
%

  \item If \(q=-2\), then
	\[
	D(n, -2)
	=\log
	\frac{n\,\Gamma(n)\,\Gamma\!\bigl(\tfrac{n-1}2\bigr)}
	     {(n+1)\,\Gamma\!\bigl(\tfrac {n+1}{2}\bigr)\,\Gamma(n-1)}
	=\log\frac{2n}{n+1}
	>0.
	\]
\end{itemize}

\textbf{Uniqueness:}
Starting from the equation $D(n,q) = 0$,
we form a difference quotient of the strictly convex function \(\log\Gamma\):
\[
\frac{\,
  \log\Gamma\!\bigl(n-1-\tfrac q2\bigr)
  -\log\Gamma\!\bigl(\tfrac{n-1-q}2\bigr)
}{n-1-\tfrac{q}2-\frac{n-1-q}2}
\;=\;
\frac{\,\log\bigl((n+1)\,\Gamma(n-1)\bigr)
        -\log\bigl(n\,\Gamma(\tfrac{n-1}2)\bigr)\,}
     {\tfrac {n-1}2}.
\]
The left hand side is a difference quotient of the strictly convex function \(\log\Gamma\) and so is strictly increasing as the value \(-q>0\) increases. Thus it can equal the right side (which is independent of $q$) for at most one $q$-value.
\end{proof}

In \autoref{fig:compare_n51020}, when $n$ is very large the curve is almost flat and then drops down almost vertically when $p$ is close to $-2$. This suggests the presence of a limiting value of $p$ when $n$ goes to infinity. We give the proof in the following lemma.

\begin{lemma}\label{lemma:3dto-2}
	For each fixed $q\in(-2,0)$, $\lim_{n\to \infty} f(n, q) = -2$.
\end{lemma}
\begin{proof}
By using the definition of $f(n, q)$,
\begin{align}
\lim_{n\to\infty} f(n, q)
= \frac{q \,\log 2}
       {\lim_{n\to\infty} \left(\log\bigl(\Gamma(n-1-\tfrac q2)\,\Gamma(\frac{n-1}{2})\bigr)
        -\log\bigl(\Gamma(\frac{n-1-q}2)\,\Gamma(n-1)\bigr) \right)}. \label{star}
\end{align}

By Stirling's ratio expansion \cite[\S5.11(iii)]{DLMF},
\begin{align*}\label{eq:stirling}
\frac{\Gamma(z+a)}{\Gamma(z+b)}
\sim z^{\,a-b}\Bigl[1 + \frac{(a-b)(a+b-1)}{2z} + \cdots\Bigr] \quad \text{as } z\to \infty.
\end{align*}

Apply this formula twice:

\begin{enumerate}
\item With \(z=n\), \(a=-1-\tfrac q2\), \(b=-1\):
\[
\frac{\Gamma(n-1-\tfrac q2)}{\Gamma(n-1)}
\sim n^{-\tfrac q2}\Bigl[1 + O(n^{-1})\Bigr].
\]

\item With \(z=\tfrac n2\), \(a=-\tfrac12\), \(b=-\tfrac12-\tfrac q2\):
\[
\frac{\Gamma(\tfrac{n-1}2)}{\Gamma(\tfrac{n-1-q}2)}
\sim \left(\tfrac n2\right)^{\! q/2}
\Bigl[1 + O(n^{-1})\Bigr].
\]
\end{enumerate}

Hence in \eqref{star} the denominator equals
\begin{align} \label{eq:denom}
\lim_{n\to \infty} \left(\log\bigl(n^{-q/2}\bigr)
+ \log\left(\frac n2\right)^{\! q/2} \right)
= -\frac q2\log 2,
\end{align}
so
\[
\lim_{n\to\infty} f(n, q) 
= -2.
\]
\end{proof}

We establish some properties of the asymptotic value $q^*(n)$ that can also be seen graphically in \autoref{fig:compare_n51020}.

\begin{lemma}\label{lemma:3dto0}
When $n\to \infty$, $q^*(n)$ as defined in \autoref{lm:root_of_D} approaches $0$.
\end{lemma}

\begin{proof}
	For arbitrary fixed $-2<q<0$, when $n$ is large, the denominator $D(n,q)\sim -\tfrac{q}{2} \log{2} >0$ from \autoref{eq:denom}. When $q=0$, we have $D(n,0)<0$ from \autoref{eq:D(q,0)}. By the Intermediate Value Theorem, for sufficiently large $n$, $q^*(n)$ is inbetween the arbitrarily chosen $q \in (-2, 0)$ and 0, which implies that $q^*(n)$ goes to $0$ when $n$ goes to infinity.

\end{proof}

\begin{lemma}
	$q^*(n+2) \ge q^*(n)$, so that the even and odd subsequences of $q^*(n)$ are increasing when $n\ge 2$.
\end{lemma}

\begin{proof}
Rewrite
\[
D(n, q) = \log \frac{n}{n+1} + \log \frac{\Gamma\left(n-1 - \frac{q}{2}\right)}{\Gamma(n-1)} - \log \frac{\Gamma\left(\frac{n-1 - q}{2}\right)}{\Gamma\left(\frac{n-1}{2}\right)}.
\]

Suppose $-2<q<0$ and $n\ge 2$. We have
\begin{align*}
&D(n+2,q)-D(n,q)\\
&= \log\frac{n+2}{n+3} + \log\frac{\Gamma\bigl(n+1-\tfrac q2\bigr)}{\Gamma(n+1)}
-\log\frac{\Gamma\bigl(\tfrac{n+1-q}2\bigr)}{\Gamma\bigl(\tfrac{n+1}2\bigr)}
-\log\frac{n}{n+1}\\
&\quad-\log\frac{\Gamma\bigl(n-1-\tfrac q2\bigr)}{\Gamma(n-1)}
+\log\frac{\Gamma\bigl(\tfrac{n-1-q}2\bigr)}{\Gamma\bigl(\tfrac{n-1}2\bigr)}\\
&= \log\left(1+\frac{2}{n(n+3)} \right)
+\log \left( 1 + \frac{q^2+2q}{4n(n-1-q)}\right) .
\end{align*}
Hence $D(n+2,q)-D(n,q)>0$, because
\begin{align*}
&\left(1+\frac{2}{n(n+3)} \right) \left( 1 + \frac{q^2+2q}{4n(n-1-q)}\right) \\
&= 1+ \frac{8n(n-1-q)+(n(n+3)+2)(q^2+2q)}{4n^2(n+3)(n-1-q)}\\
&> 1 + \frac{8n(n-1)-n(n+3)-2}{4n^2(n+3)(n-1-q)}\\
& \geq 1 + \frac{n(7n-12)}{4n^2(n+3)(n-1-q)}\\
&> 1.
\end{align*}
By substituting $q=q^*(n)$ and using that $D(n,q^*(n))=0$ we deduce $D(n+2, q^*(n))>0$, which implies $q^*(n+2)>q^*(n)$.

\end{proof}

\subsection{Irregular symmetry-breaking sets in three dimensions}\label{sec:3d-algo}

As in \autoref{sec:intersection} for two dimensions, there are shapes in three dimensions whose capacity ratio surpasses that of the ball for certain parameters $r$ and $s$. In two dimensions, the vertex sets of regular odd-sided polygons form such a family naturally. In three dimensions no analogous family is apparent, but the five-point and six-point configurations in Figures \ref{fig:pts5} and \ref{fig:pts6} give explicit examples. The five-point set splits one vertex of a regular tetrahedron symmetrically into two points on the unit sphere. The six-point set is a regular pentagonal pyramid with apex at the north pole and base at height $-3/10$. Explicitly,
\[
\begin{aligned}
K_6&=\{(0,0,1)\}\cup
\{(\rho\cos\theta_j,\rho\sin\theta_j,-3/10):j=0,\ldots,4\},\\
\rho&=\frac{\sqrt{91}}{10},\qquad \theta_j=\frac{2\pi j}{5}.
\end{aligned}
\]

\begin{figure}[htbp]
  \centering
  \setlength{\captionindent}{0pt}
  \begin{minipage}[t]{0.45\textwidth}
    \centering
    \includegraphics[width=\textwidth]{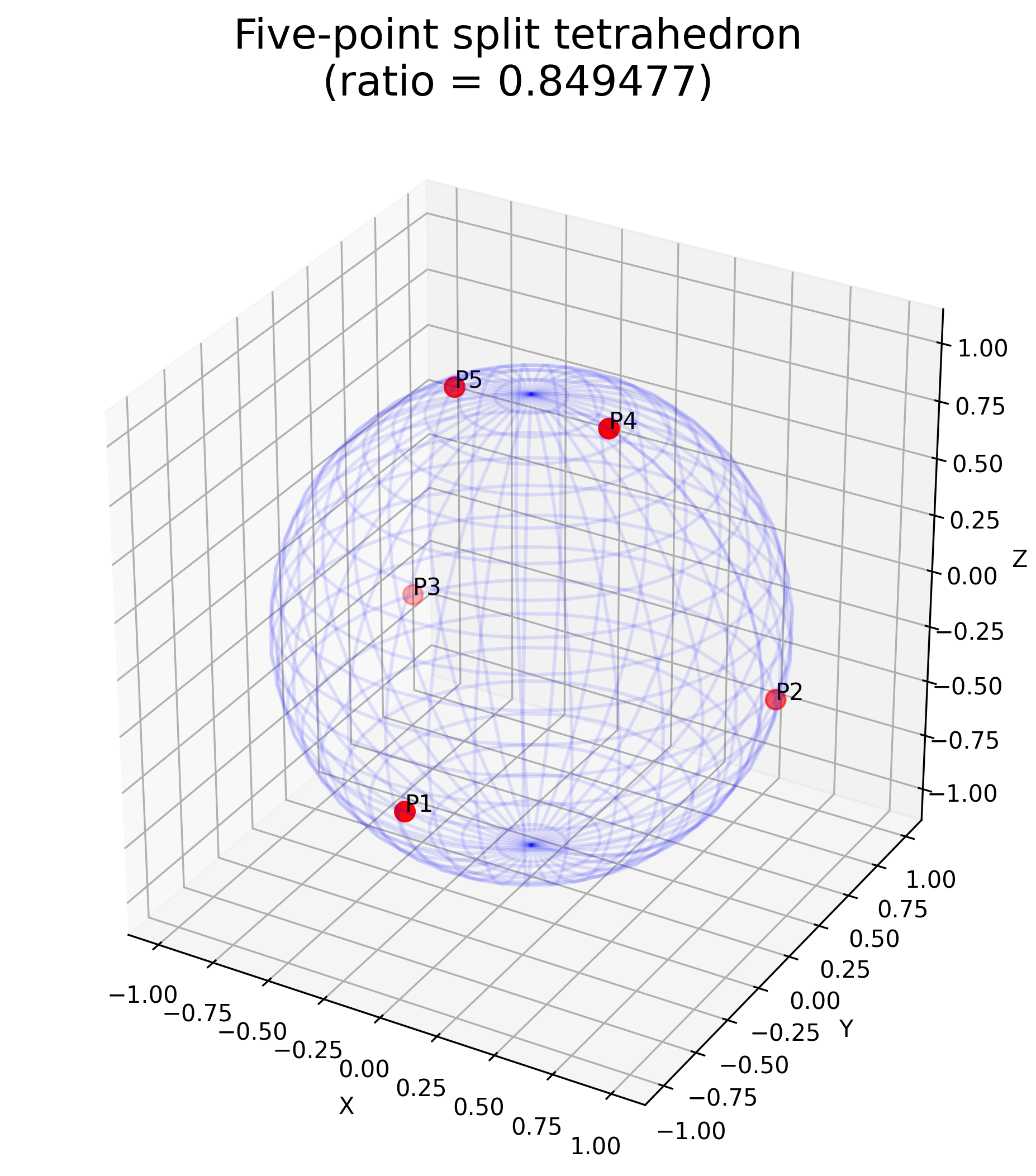}
    \caption{A $5$-point set with higher capacity ratio than the ball when $r=3$ and $s=1.08$. The points are
    $(\pm 1/3,0,2\sqrt{2}/3)$, $(0,-2\sqrt{2}/3,-1/3)$, and $(\pm\sqrt{6}/3,\sqrt{2}/3,-1/3)$. Its capacity ratio is approximately $0.849477$, exceeding the value $0.844721$ for the ball. \label{fig:pts5}}
  \end{minipage}
  \hspace{0.05\textwidth}
  \begin{minipage}[t]{0.45\textwidth}
    \centering
    \includegraphics[width=\textwidth]{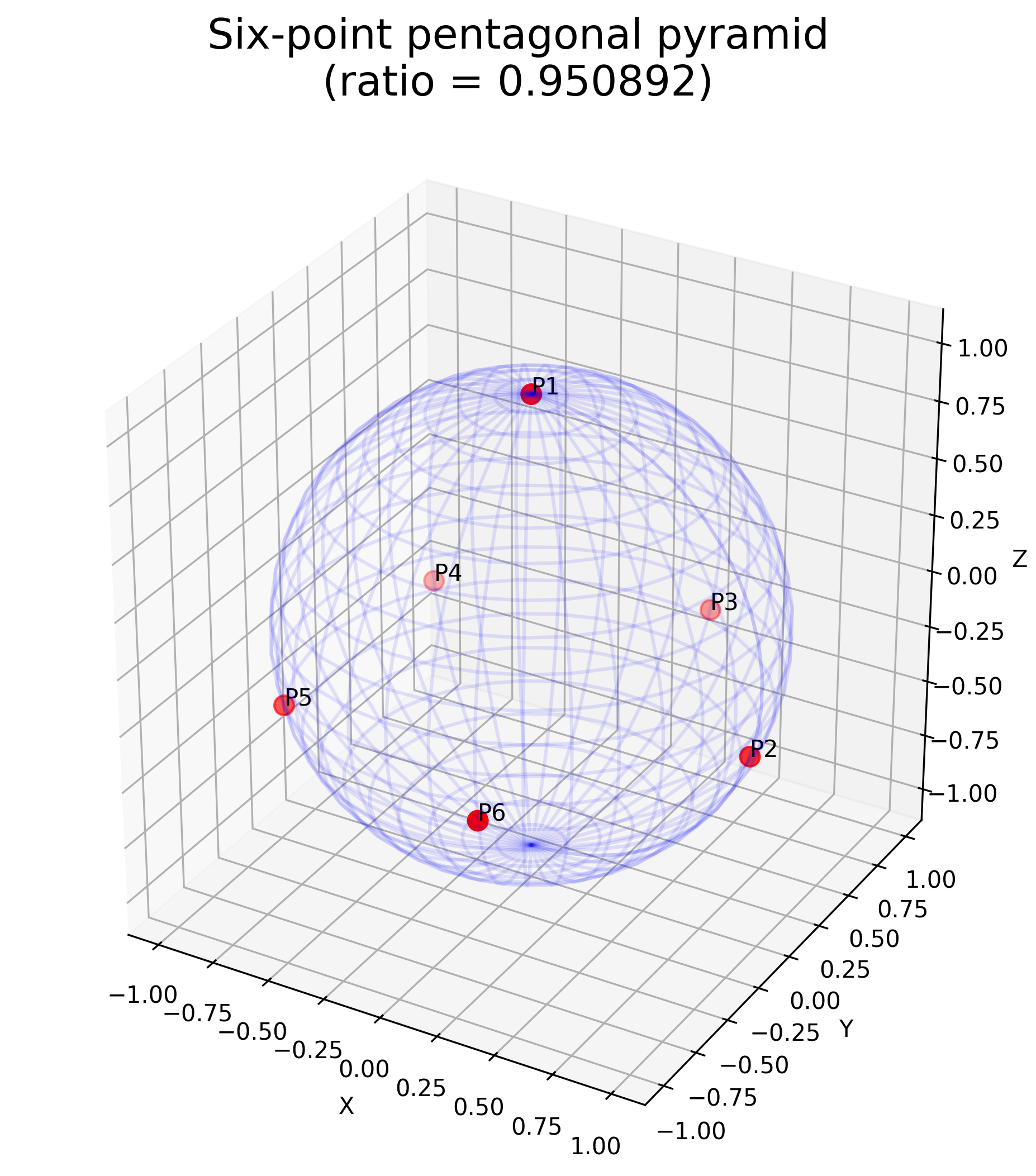}
    \caption{A $6$-point regular pentagonal pyramid with higher capacity ratio than the ball when $r=2.2$ and $s=1.54$. Its apex is $(0,0,1)$; its five base vertices have height $-3/10$ and radius $\sqrt{91}/10$. Its capacity ratio is approximately $0.950892$, exceeding the value $0.945824$ for the ball. \label{fig:pts6}}
  \end{minipage}
\end{figure}

\FloatBarrier 
A second pattern mirrors the two-dimensional case in \autoref{sec:intersection}. The curves in \autoref{fig:capratio_comparison_higher_dim} indicate where the capacity ratios of the regular simplex (4 points), the five-point set, and the six-point set equal that of the ball. All three pass through $(-2,-2)$. In the close-up in \autoref{fig:capratio_comparison_higher_dim}(a), they appear nested, with the six-point curve highest. The wider view in \autoref{fig:capratio_comparison_higher_dim}(b), however, reveals that the regular-simplex curve intersects both other curves. The numerical crossings with the five- and six-point curves occur near $(-6.114777,-0.814685)$ and $(-9.858871,-0.750228)$, respectively. Thus, as in two dimensions, increasing the number of points does not uniformly enlarge the observed region where a configuration beats the ball. This crossover behavior remains to be understood.

The regular-simplex curve is given analytically by \autoref{eq:curve} with $n=3$. For the fixed five- and six-point sets, we compute capacities on a prescribed grid of $r$-values and use a bracketed root solver to find $s$ where the configuration and ball have equal capacity ratios. For $r>2$, Bj\"orck's support theorem \cite[Theorem 12]{B56} reduces the denominator calculation to supports of at most four points. Computational details are available in the \href{https://github.com/vhdvhd/Riesz-capacity-ratios-with-negative-exponents}{code}.

\begin{figure}[htbp]
    \centering
    \includegraphics[width=\textwidth]{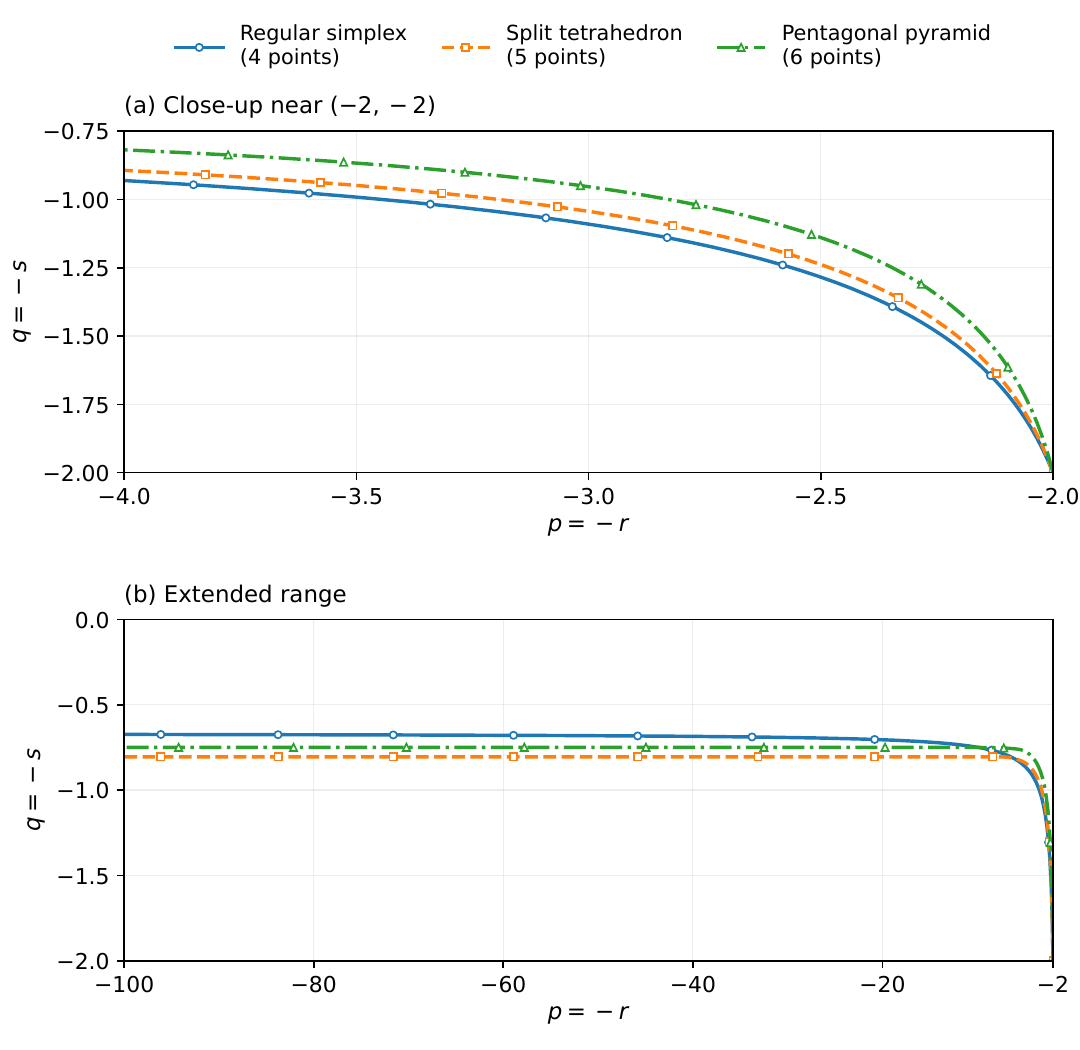}
    \caption{Capacity-ratio equality curves for the regular simplex and the configurations in \autoref{fig:pts5} and \autoref{fig:pts6}. (a)~Close-up near $(-2,-2)$. (b)~Extended view with $-100\le p\le-2$ and $-2\le q\le0$, revealing intersections analogous to those in two dimensions: the regular-simplex curve crosses both other curves. Sampled checks favor the ball just above each curve and the corresponding finite set just below. All three curves meet at $(-2,-2)$. \label{fig:capratio_comparison_higher_dim}}
\end{figure}


\FloatBarrier 
The high-dimensional picture should also extend into the second quadrant. Burchard, Choksi and Hess-Childs \cite{BCH20} proved that, for each fixed $q>0$, the ball in $\bbR^n$ fails to maximize $\capq(K)/\operatorname{diam}(K)$ when $n$ is sufficiently large. Recall that diameter corresponds to $p=-\infty$. Locating the symmetry-breaking boundary for positive $q$ remains a challenge because accurate capacity computations are difficult.

\FloatBarrier
\section*{Acknowledgments}
The research in this paper forms part of the author's PhD dissertation at the University of Illinois Urbana--Champaign \cite{F25}.

\section*{Declarations}
\subsection*{Funding details}
This work was supported by the National Science Foundation under grant number 2246537 to Richard Laugesen.
\subsection*{Disclosure statement}
There are no competing interests to declare. 
\subsection*{Disclosure of generative AI use}
During the preparation of this manuscript, the author used OpenAI's GPT-5.6 Sol and GPT-6 Astra (Ultra reasoning setting) to assist with proofreading, reviewing and explaining the associated code, and making minor corrections to the text and code. These tools were used to improve readability and support code understanding and checking. The author takes responsibility for the content of the manuscript and the associated code.
\subsection*{Supplemental online material}
The code for the experiments in \autoref{chp:1d}, \autoref{chp:2d} and \autoref{chp:3d} is available on github: \\
\url{https://github.com/vhdvhd/Riesz-capacity-ratios-with-negative-exponents}.

\bibliographystyle{plain}

\end{document}